\documentclass[12pt,a4paper]{article}
\usepackage{a4wide}
\usepackage{amsfonts,amsmath,latexsym,amssymb,euscript,graphicx,units,mathrsfs,color}

\usepackage{amsthm, amsfonts, amssymb,eufrak}

\usepackage{float}
\newfloat{figure}{H}{lof}
\floatname{figure}{\figurename}

\DeclareMathAlphabet{\eufrak}{U}{}{}{}  
\SetMathAlphabet\eufrak{normal}{U}{euf}{m}{n}
\SetMathAlphabet\eufrak{bold}{U}{euf}{b}{n}

\usepackage{amsmath, amsthm, amsfonts, amssymb}

\numberwithin{equation}{section}

\newenvironment{Proof}{\removelastskip\par\medskip
\noindent{\em Proof.}
\rm}{\hfill$\square$\\\par\medbreak}

\def\real{{\mathord{{\rm I\kern-2.8pt R}}}}        
\def\inte{{\mathord{{\rm I\kern-2.8pt N}}}}
\def\PP{{\mathord{{\rm I\kern-2.8pt P}}}}

\def\real{{\mathord{\mathbb R}}}

\def\inte{{\mathord{\mathbb N}}}

\def\R{\right}
\def\L{\left}

\newcommand{\e}{\varepsilon}

\def\P{\mathbb{P}}
\def\E{\mathop{\hbox{\rm I\kern-0.20em E}}\nolimits}

\def\Q{\mathbb{Q}}

\newtheorem{prop}{Proposition}[section]

\newtheorem{lemma}[prop]{Lemma}

\newtheorem{theorem}[prop]{Theorem}
\newtheorem{remark}[prop]{Remark}

\def\e{\varepsilon}
\def\de{\delta}

\author{{Anthony R\'eveillac}\footnote{CEREMADE, CNRS UMR 7534, Place du Mar\'echal De Lattre De Tassigny,
75775 PARIS CEDEX 16 - FRANCE, anthony.reveillac@ceremade.dauphine.fr}\\ Universit\'e Paris Dauphine and Humboldt-Universit\"at zu Berlin }
\title{Weak martingale representation for continuous Markov processes and application to quadratic growth BSDEs}

\allowdisplaybreaks

\begin{document}

\maketitle
 
\begin{abstract}
\noindent
In this paper we prove that every random variable of the form $F(M_T)$ with $F:\real^d \to\real$ a Borelian map and $M$ a $d$-dimensional continuous Markov martingale with respect to a Markov filtration $\mathcal{F}$ admits an exact integral representation with respect to $M$, that is, without any orthogonal component. This representation holds  true regardless any regularity assumption on $F$. We extend this result to Markovian quadratic growth BSDEs driven by $M$ and show they can be solved without an orthogonal component. To this end, we extend first existence results for such BSDEs under a general filtration and then obtain regularity properties such as differentiability for the solution process.       
\end{abstract} 


\section{Introduction}

One of the most useful and striking property in stochastic calculus is probably the \textit{martingale representation property} (MRP). Given a $d$-dimensional martingale $M:=(M^1,\cdots,M^d)$ with respect to a filtered probability space $(\Omega,\mathcal{F}_T,\mathcal{F}:=(\mathcal{F}_t)_{t\in [0,T]}),\P)$, we say that $M$ enjoys MRP if for every $\mathcal{F}$-(local) martingale $Y$, there exists an integrable predictable process $Z$ such that $Y$ can be decomposed as:
\begin{equation}
\label{eq:intro:repr}
Y=Y_0+\int_0^\cdot Z_s dM_s.
\end{equation} 
Alternatively, MRP entails that for every integrable $\mathcal{F}_T$-measurable random variable $\zeta$, there exists an integrable predictable process $Z$ such that 
\begin{equation}
\label{eq:intro:Clark}
\zeta = Y_0+\int_0^T Z_s dM_s,
\end{equation}
which is a direct consequence of \eqref{eq:intro:repr} by representing the martingale $Y:=\E[\zeta\vert \mathcal{F}_\cdot]$. This second formulation is well-known under the name of Clark-Haussmann-Ocone formula. 
Relations of the type \eqref{eq:intro:repr}-\eqref{eq:intro:Clark} are very useful in applications, like for example in Financial Mathematics, where $\zeta$ represents a contingent claim and $Z$ a strategy which loosely speaking allows one to replicate "optimally" (in some sense to be precised) in a dynamic and predictable way this claim. MRP is a very strong property and unfortunately usually fails to hold for a given martingale $M$. Indeed, according to the, by now, classical theory (see \textit{e.g.} \cite[Theorem 4.6]{RevuzYor} or \cite{Jacod}), MRP is basically equivalent to the fact that $\P$ is an extreme point of the set of martingale measures for $M$. Since this is usually not the case, a relation of the form \eqref{eq:intro:repr} can not hold true for a given $\mathcal{F}$-martingale $Y$ and one as to consider in addition of $Z$, a martingale $N$ strongly orthogonal to $M$ (\textit{i.e.} $\langle M^i, N \rangle=0, \;i=1,\cdots,d$) such that \eqref{eq:intro:repr} is replaced with the so-called Galtchouk-Kunita-Watanabe decomposition:
\begin{equation}
\label{eq:intro:reprbis}
Y=Y_0+\int_0^\cdot Z_s dM_s + N.
\end{equation}
Note, besides, that the absence of MRP for a martingale $M$ is not a quantitative statement, that is, we do not know a priori which are the martingales $Y$ for which the component $N$ is really needed. This remark leads to the following question. Can we characterize the martingales $Y$ (or the random variables $\zeta$) on $(\Omega,\mathcal{F}_T,\mathcal{F}=(\mathcal{F}_t)_{t\in [0,T]}),\P)$ for which a representation of the form \eqref{eq:intro:repr} (or \eqref{eq:intro:Clark}) holds? Or, at least, can we provide a class of martingales $Y$ or of variables $\zeta$ which fulfill \eqref{eq:intro:repr}-\eqref{eq:intro:Clark}? An even more complex question, related to the first one, is to study a generalization of martingale representation property, namely to study existence/uniqueness/regularity of solutions of Backward Stochastic Differential Equations (BSDEs), which in our context, given an $\mathcal{F}_T$-measurable random variable $\zeta$ and a predictable process $f:[0,T]\times\real\times\real \to \real$, consists in finding a triple $(Y,Z,N)$ such that the following equation is satisfied:
\begin{equation}
\label{eq:BSDEIntro}
Y_t=\zeta +\int_t^T f(s,Y_s,Z_s) d \langle M, M \rangle_s -\int_t^T Z_s dM_s - \int_t^T dN_s, \quad \forall t\in [0,T]
\end{equation}
where $Y$ and $Z$ are predictable processes and $N$ is a martingale strongly orthogonal to $M$ (here for simplicity we wrote the equation for $d=1$). One is then interested, in giving conditions on the data of the equation, namely, the \textit{terminal condition} $\zeta$ and on the \textit{generator (or driver)} $f$, under which the orthogonal component of the solution $N$ vanishes.  
\\\\
\noindent                       
The objective of this paper is three fold. First for specific martingales $M$, we provide a class of martingales $Y$ or of random variables $\zeta$ for which their representation with respect to $M$ holds without any orthogonal component like in \eqref{eq:intro:repr} or \eqref{eq:intro:Clark}. More precisely, assuming that $M$ is a square-integrable martingale which is at the same time a strong Markov process with respect to a filtration $(\mathcal{F}_t)_{t\in [0,T]}$, we prove in Theorem \ref{th:main} that every integrable random variable of the form $F(M_T)$ where $F:\real^d \to \real$ is \textit{just} a Borelian function admits a representation of the form \eqref{eq:intro:Clark}, and also that \eqref{eq:intro:repr} holds for the martingale $Y:=\E[F(M_T) \vert \mathcal{F}_\cdot]$. At this point, we stress that no regularity assumption whatsoever is assumed on the map $F$. We will say that $M$ admits \textit{a weak martingale representation property} since all martingales of the form $Y:=\E[F(M_T) \vert \mathcal{F}_\cdot]$ admits an integral representation against $M$ without an orthogonal component. In a second time, we turn to the same type of properties for quadratic growth BSDEs (qgBSDEs for short), meaning that the map $f$ in \eqref{eq:BSDEIntro} has quadratic growth in the $z$ variable. Before, proving that the component $N$ of the solution vanishes (under some conditions) we have to fill a gap in the existence theory of such BSDEs and we have to prove existence of solutions in that context. Note that this is not covered by the literature up to now, since the only existence results in this area are those obtained by Morlais \cite{Morlais} and very recently by Barrieu and El Karoui \cite{BarrieuElKaroui} under the assumption that the filtration $\mathcal{F}$ is continuous (so there are no discontinuous martingales on such spaces). This assumption is really needed in both mentioned papers. Hence in our second main result: Theorem \ref{th:existence}, we fill this gap and prove that under a general right-continuous filtration $\mathcal{F}$, equations of the form \eqref{eq:BSDEIntro} admit at least one solution $(Y,Z,N)$ (in the good space) when the terminal condition $\zeta$ is a bounded random variable and the driver $f$ has quadratic growth in $z$. Note that for this property we simply assume that $M$ is a continuous martingale under $\mathcal{F}$ and in particular we do not assume that $M$ is a Markov process. We are able to prove this result by combining arguments of \cite{BarrieuElKaroui} and of \cite{Morlais}, and by replacing a monotone stability result obtained in \cite{BarrieuElKaroui} for a special class of continuous semimartingales with a compactness type argument derived by Barlow and Protter in \cite{BarlowProtter} valid for general semimartingales. Finally, the third main result of this paper is to show in Theorem \ref{theorem:main} that in a continuous Markovian setting (\textit{i.e.} $M$ is a continuous martingale and a Markov process with respect to $\mathcal{F}$; and the terminal condition $\zeta$ is of the form $F(X_T,M_T)$ with $F$ any bounded Borelian map, and $X$ denotes the strong solution of an SDE driven by $M$), the solution $N$ vanishes (we refer to Section \ref{section:without} for a precise statement). This property requires additional results on the regularity of the solution $(Y,Z,N)$ (given in Sections \ref{sub:Markov} and \ref{sub:Regul}) which once again are not contained in the literature.\\\\               
\noindent
We would like to make some comments about results in the literature. While we were writing the Note \cite{ReveillacOrtho} which was a pre-version of the present paper, we realized that related results have been obtained in the literature. For instance we mention the paper by Jacod, M\'el\'eard and Protter \cite{JacodMeleardProtter} where the authors prove (among other things) a Clark-Haussmann representation formula for random variables of the form $F(M_T)$ where $M$ is a c\`adl\`ag Markov martingale and $F$ is a deterministic map regular enough. Basically the Markov setting allows one to represent the martingale $Y:=\E[F(M_T) \vert \mathcal{F}_\cdot]$ as a deterministic function $u$ of time and $M$ (\textit{i.e.}, $Y_t=u(t,M_t)$). Then the smoothness on $F$, transfers to $u$ so that one can deduce that the orthogonal component $N$ in \eqref{eq:intro:reprbis} vanishes. In \cite{JacodMeleardProtter}, the authors basically assume that $F$ is such that $u$ is differentiable in time and twice differentiable in space. As we will see, in the continuous case this regularity is not needed and especially the regularity in time. Our method is also based on the representation of $Y$ as $u(t,M_t)$ but our analysis differs from the one presented in \cite{JacodMeleardProtter}. Another technology presented in the literature consists in combining the Markovian structure and the Malliavin calculus for some particular c\`adl\`ag Markov martingales to get an exact representation for $F(M_T)$ but once again under some regularity properties for $F$ (especially if $M$ is continuous), we refer to the monograph by Privault \cite[Section 3.7]{PrivaultLN}.\\\\
\noindent
We proceed as follows. First in Section \ref{section:preliminaries} we present the main notations and definitions that we will be used in our framework. Then, we derive in Section \ref{section:wmr} the representation property for a continuous Markov martingale. Then in Section \ref{section:existence} we give an existence result for a qgBSDEs driven by a continuous martingale with respect to a \textit{not necessarily continuous} filtration $\mathcal{F}$. Finally, we prove in Section \ref{section:BSDEwithout} that such BSDEs can be solved without any orthogonal component in a Markovian context for general terminal conditions.        

\section{Preliminaries}
\label{section:preliminaries}

Fix $T$ in $(0,\infty)$. Let $M:=(M_t)_{t\in [0,T]}$ be a $d$-dimensional continuous square integrable martingale ($d\geq 1$) with respect to a right-continuous completed filtration $(\mathcal{F}_t)_{t\in [0,T]}$ (so satisfying the usual conditions), both defined on a probability space $(\Omega,\mathcal{F},\P)$. The expectation with respect to $\P$ will be denoted by $\E$. In the following we will assume that $\mathcal{F}=\mathcal{F}_T$. The Kunita-Watanabe inequality implies that there exists a $\real^{d \times d}$-valued predictable process $q:=(q_t)_{t\in [0,T]}$ such that $ [ M, M ]_t=\int_0^t q_s q_s^\ast dC_s, \; t\in [0,T] $
with $C:=\arctan(\sum_{i=1}^d [ M^{(i)}, M^{(i)} ])$ (for more details we refer to \cite{Morlais}) and where $q^\ast$ denotes the transpose of the matrix $q$. Throughout this paper, $[P_1,P_2]$ will denote the quadratic co-variations between two semi-martingales $P_1$ and $P_2$, and $\langle P_1,P_2\rangle$ stands for its compensator. For a given martingale $P$, we denote by $P^c$ the continuous martingale part of $P$. In addition if $P$ is continuous, we denote by $\mathcal{E}(P)$ the stochastic exponential of $P$, \textit{i.e.} the stochastic process defined as: $\mathcal{E}(P):=\exp(P-\frac12 \langle P, P \rangle)$. Remark that by definition, $q^\ast q$ is a positive semidefinite matrix and thus by Cholesky's decomposition we can assume that $q$ is a lower triangular matrix with non-negative diagonal entries meaning that $q$ itself is positive semidefinite. We use the notation $|\cdot|$ for the Euclidian norm on $\real^d$.
We introduce several spaces of interest in our context. For any $p>1$, we denote by $\mathcal{S}^p$ the set of one-dimensional predictable processes $Y$ such that $\E[(\sup_{t\in [0,T]} |Y_t|^2)^{p/2}]<\infty$, by $\mathcal{H}^p$ the set of $d$-dimensional predictable processes $Z$ such that $\E\left[\left(\int_0^T |Z_s q_s^\ast|^2 dC_s\right)^{p/2} \right]<\infty$, and by $\mathcal{O}^p$ the space of one-dimensional c\`adl\`ag martingales $N$ strongly orthogonal to $M$ (\textit{i.e.} $\langle M^{i}, N \rangle=0, \; i=1,\ldots,d$) such that $N_0=0, \; \P$-a.s. and such that $\E[[N]_T^{p/2}]<\infty$. Note that since $M$ is continuous, we also have for $N$ in $\mathcal{O}^p$ that $[M^{i}, N]=0, \; i=1,\ldots,d$. 
By Galtchouk-Kunita-Watanabe's decomposition, every square integrable c\`adl\`ag martingale on $(\Omega,\mathcal{F},(\mathcal{F}_t)_{t\in [0,T]},\P)$ starting from zero at time zero can be decomposed as: 
$$ \int_0^\cdot Z_s dM_s +N = \sum_{i=1}^d \int_0^\cdot Z_s^i dM^i_s +N $$ 
where $Z=(Z^1,\ldots,Z^d)$ is a $d$-dimensional process in $\mathcal{H}^2$ and $N$ belongs to $\mathcal{O}^2$. In this paper, by martingale we always mean a martingale with respect to the filtration $(\mathcal{F}_t)_{t\in [0,T]}$.  
\\\\
\noindent
At the exception of Section \ref{section:existence}, we will always assume in addition, that $M$ is a strong Markov process with respect to $(\mathcal{F}_t)_{t\in [0,T]}$ and that the filtration is a Markov filtration in the sense of \cite[(3.4)]{CinlarJacodProtterSharpe}. For $(t,m)$ in $[0,T]\times \real$ we denote by $M^{t,m}$ the process $M$ starting at $m$ at time $t$ defined as $ M_s^{t,m}:=m+M_s-M_t, \; s \in [t,T].$ For any stochastic process $\alpha=(\alpha_t)_{t\in [0,T]}$, we write $\alpha \equiv 0$ for $\alpha_t=0, \; d\P\otimes dC_t-a.s.$.\\\\
\noindent 
Throughout this paper, $c$ will denote a positive constant which can differ from line to line. Given non-negative integers $p, q, r$, we set $C^{p,q}([0,T] \times \real^{r})$ the set of Borelian functions $u:[0,T]\times \real^r \to \real$ such that the map $ [0,T] \ni t \mapsto u(t,x)$ is $p$ times continuously differentiable for every element $x$ in $\real^r$, and the map $ \real^r \ni x \mapsto u(t,x)$ is $q$ times continuously differentiable for every $t$ in $[0,T]$. We finally introduce the space of BMO martingales. A c\`adl\`ag martingale $P$ is said to be a BMO martingale is the there exists a constant $\alpha>0$ such that $|\Delta P|^2\leq \alpha $ (where $\Delta P$ denotes the jumps of $P$) and 
$$ \textrm{esssup}_{\tau\in [0,T]}\E[[P]_T-[P]_\tau] \leq \alpha, \quad \P-a.s. $$
where the essential supremum is taken over all the $\mathcal{F}$-stopping times $\tau$ in $[0,T]$. 

\section{Weak martingale representation property for continuous Markov processes}
\label{section:wmr}

This section is devoted to prove the following theorem which constitutes one of the main results of this paper. We use the definitions, assumptions and notations of Section \ref{section:preliminaries}.

\begin{theorem}
\label{th:main}
Let $F:\real^d \to\real$ be a Borelian function such that $F(M_T)$ is a square integrable $\mathcal{F}_T$-measurable random variable. Let $Y$ be the square integrable martingale defined as $Y:=\E[F(M_T)\vert \mathcal{F}_\cdot]$. Then $Y$ admits the following martingale representation (without orthogonal component):
$$ Y= Y_0 + \int_0^\cdot Z_s dM_s $$
where $Z$ is an element of $\mathcal{H}^2$. 
\end{theorem}

\noindent
We will say that $M$ admits a \textit{weak martingale representation property}. The proof of Theorem \ref{th:main} will be given at the end of this Section and requires several intermediary results including the following Lemma which is a key result in our approach and which is a refinement of the main result of the Note \cite{ReveillacOrtho}.

\begin{lemma}
\label{lemma:rep1}
Let $m\geq 1$ be an integer. Let $L:=(L_t)_{t\in [0,T]}$ be a square integrable continuous semimartingale with values in $\real^m$. Let $(Y_t)_{t\in [0,T]}$ be a one-dimensional c\`adl\`ag semimartingale with decomposition:
\begin{equation}
\label{eq:rep2}
Y_t=Y_0+\int_0^t Z_s dL_s + N_t + A_t, \quad t\in [0,T] 
\end{equation}
where $Z$ is a predictable process such that the stochastic integral $\int_0^\cdot Z_s dM_s$ makes sense and is a local martingale, $N$ is a one-dimensional square integrable martingale satisfying $[L^i,N]=0, \; i \in \{1,\ldots,m\}$ and $A$ is a continuous predictable process with finite variation. If there exists a bounded Borelian deterministic function $u:[0,T] \times \real^m \to \real$ with $x \mapsto u(t,x)$ differentiable for every element $t$ in $[0,T]$, with derivative $\partial_x u$ continuous in $(t,x)$, such that $Y=u(\cdot,L_\cdot)$. Then
$N \equiv 0$.
\end{lemma}

\begin{proof} 
First note that since $u$ is a Borelian function and since $L$ is a predictable process (since continuous), the process $Y=u(\cdot,L)$ is a predictable process. Then by \cite[Proposition I.2.24]{JacodShiryaev} the jump times of $Y$ are predictable times. Since the later are exactly the jump times of $N$ it entails that $N$ is a continuous martingale (see \cite[Corollary I.2.31]{JacodShiryaev}). This remark is crucial in our proof where we mimic a technique used in \cite{IRR} and in the Note \cite{ReveillacOrtho}. The core idea is the following. Since $N$ is continuous, $N\equiv 0$ if and only if $[N,N]\equiv 0$. By \eqref{eq:rep2} it holds that $[N,N]=[Y,N]$, which can be computed using the Markovian representation of $Y$.\\
\noindent 
Let $\pi^{(n)}:=\{0=t_0^{(n)} \leq t_1^{(n)} \leq \cdots \leq t_N^{(n)}=T\}$ be a sequence of subdivisions of $[0,T]$ whose mesh $\vert \pi^{(n)} \vert$ tends to zero as $n$ goes to the infinity such that
\begin{equation}
\label{eq:bracket}
\lim_{n\to \infty} \sup_{t\leq s \leq T} \left\vert [ Y, N ]_s-\sum_{j=0}^{\varphi_s-1} (u(t_{j+1}^{(n)},L_{t_{j+1}^{(n)}})-u(t_j^{(n)},L_{t_j^{(n)}})) \Delta_j N \right\vert =0
\end{equation}
where the limit is understood in probability with respect to $\P$, $\Delta_j N:=N_{t_{j+1}^{(n)}}-N_{t_j^{(n)}}$ and $\varphi_s=j$ such that $t_j^{(n)} \leq s < t_{j+1}^{(n)}$. For simplicity we will drop the superscripts ${(n)}$ in the rest of the proof except when its absence could lead to a confusion. 
We will show that $[ Y, N ]\equiv 0$. We have that 
\begin{eqnarray}
\label{bracket1}
[Y,N]_s&=&\sum_{j=0}^{\varphi_s-1} (u(t_{j+1},L_{t_{j+1}})-u(t_j,L_{t_j})) \Delta_j N\nonumber\\
&=& \sum_{j=0}^{\varphi_s-1} \bigg( (u(t_{j+1},L_{t_j})-u(t_j,L_{t_j})) \Delta_j N\nonumber\\
&&+\sum_{j=0}^{\varphi_s-1} (u(t_{j+1},L_{t_{j+1}})-u(t_{j+1},L_{t_j})) \Delta_j N \bigg)\nonumber\\
&=:&A_{s,1}^{(n)}+A_{s,2}^{(n)}.
\end{eqnarray}
We consider the two summands above separately. We start with the term $A_2$ and we prove that
\begin{equation}
\label{eq:secondterm}
\lim_{n\to \infty} \sup_{0\leq s \leq T} |A_{s,2}^{(n)}| = 0, \textrm{ in } \P-\textrm{probability.}.
\end{equation}
For $i$ in $\{1,\ldots,d\}$, we denote by $M^i$ the $i$th component of $M$. 
We have  
\begin{eqnarray}
\label{bracket2}
A_{s,2}^{(n)}&=&\sum_{j=0}^{\varphi_s-1} (u(t_{j+1},L_{t_{j+1}})-u(t_{j+1},L_{t_j})) \Delta_j N\nonumber\\
&=&\sum_{j=0}^{\varphi_s-1} \sum_{i=1}^m \Delta_j N \big(u(t_{j+1},L_{t_j}^1,\ldots,L_{t_j}^{i-1},L_{t_{j+1}}^i,\ldots,L_{t_{j+1}}^m)\nonumber\\
&& \quad \quad \quad \quad -u(t_{j+1},L_{t_j}^1,\ldots,L_{t_j}^i,L_{t_{j+1}}^{i+1},\ldots,L_{t_{j+1}}^m)\big) \nonumber\\
&=&\sum_{j=0}^{\varphi_s-1} \sum_{i=1}^m \partial_{x_i} u(t_j,L_{t_j}) \Delta_j M^i \Delta_j N + R_{s,j,n}
\end{eqnarray}
where $R_{s,j,n}$ is defined as
$$R_{s,j,n}:= \sum_{i=1}^m \Delta_j L^i \Delta_j N \big(\partial_{x_i} u(t_{j+1},L_{t_j}^1,\ldots,L_{t_j}^{i-1},\bar{X}^i,L_{t_{j+1}}^{i+1},\ldots,L_{t_{j+1}}^m)-\partial_{x_i} u(t_j,L_{t_j})\big),$$
$\bar{X}^{i}$ is a random point between $L_{t_j}^i$ and $L_{t_{j+1}}^i$, and $\Delta_j L^i:=L_{t_{j+1}}^i-L_{t_j}^i$. 
Now consider the remainder term $\sum_{j=0}^n R_{s,j,n}$. We set: 
$$ \delta^{(n)}:=\sup_{a \in [s,t], \; |t-s|\leq |\pi^{(n)}|} \sum_{i=1}^m \left\vert \partial_{x_i} u(c,(L_b^1,\ldots,L_b^{i-1},L_a^i,L_c^{i+1},\ldots,L_{c}^{d}))-\partial_{x_i} u(b,L_b)\right\vert$$
We have that
$$\left\vert \sum_{j=0}^n R_{s,j,n} \right\vert \leq \delta^{(n)} \sum_{i=1}^m \sum_{j=0}^n |\Delta_j M^{i} \Delta_j N| \leq 2 \delta^{(n)} \sum_{i=1}^m \sum_{j=0}^n |\Delta_j L^i|^2 + |\Delta_j N|^2.$$
Now by continuity of $\partial_x u$ and of $X$, it is clear that $\lim_{n\to \infty} \delta^{(n)}=0$, $\P$-a.s. whereas $\sum_{i=1}^m \sum_{j=0}^n |\Delta_j L^i|^2 + |\Delta_j N|^2$ converges in $\P$-probability to $\sum_{i=1}^m [M^i,M^i]_T + [N,N]_T$ has $n$ goes to infinity, hence
$$ \lim_{n\to \infty} \left\vert \sum_{j=0}^n R_{s,j,n} \right\vert = 0, \textrm{ in } \P-\textrm{probability.} $$ 
Then it follows using \eqref{bracket2} that 
$$\lim_{n\to \infty} A_{s,2}^{(n)} = \lim_{n\to\infty} \sum_{j=0}^{\varphi_s-1} (u(t_j,L_{t_{j+1}})-u(t_j,L_{t_j})) \Delta_j N =\sum_{i=1}^m \int_0^s \partial_{x_i} u(r,L_r) d[ L^i, N ]_s=0$$
by strong orthogonality between $L$ and $N$ which proves \eqref{eq:secondterm}. Then \eqref{eq:bracket} and \eqref{bracket1} entail that $ \lim_{n\to \infty} \sup_{0\leq s \leq T} \vert A_{s,1}^{(n)}-[ Y, N ]_s \vert = 0, \textrm{ in } \P-\textrm{probability}$ 
and so
$$ \lim_{n\to \infty} \sup_{0\leq s \leq T} \vert A_{s,1}^{(n)}-[ N, N ]_s \vert = 0, \textrm{ in } \P-\textrm{probability}.$$
We will prove that $P:=[ N, N ]$ is a martingale, since it is by definition of finite variation and continuous this will show that $[ N, N ]\equiv 0$. We know that $\E[|[ N, N ]_s|]<\infty, \; \forall s \in [0,T].$
Now fix $0\leq s_1\leq s_2\leq T$. For an element $t_j$ of the subdivisions considered above we let $\delta_j u:=u(t_{j+1},L_{t_j})-u(t_j,L_{t_j})$. We have that
\begin{align}
\label{eq:martingale}
\E[P_{s_2} \vert \mathcal{F}_{s_1}]&=\E\left[\lim_{n \to \infty} \sum_{j=0}^{\varphi_{s_2}-1} \delta_j u \; \Delta_j N \vert \mathcal{F}_{s_1}\right]=\E\left[\lim_{n \to \infty} \sum_{j=0}^{\varphi_{s_2}-1} \delta_j u \; \Delta_j N + (N_{s_2}-N_{t_{\varphi_{s_2}}}) \vert \mathcal{F}_{s_1}\right]
\end{align}
where the last equality is a consequence of the continuity of the martingale $N$. In addition the sequence of random variables $\left( \sum_{j=0}^{\varphi_{s_2}-1} \delta_j u \; \Delta_j N + (N_{s_2}-N_{t_{\varphi_{s_2}}}) \right)_{n}$ is uniformly integrable. Indeed, since the function $u$ is bounded we have that
\begin{align*}
\E\left[\left| \sum_{j=0}^{\varphi_{s_2}-1} \delta_j u \; \Delta_j N + (N_{s_2}-N_{t_{\varphi_{s_2}}}) \right|^2\right]&=\sum_{j=0}^{\varphi_{s_2}-1} \E\left[|\delta_j u|^2 |\Delta_j N|^2 + |(N_{s_2}-N_{t_{\varphi_{s_2}}})|^2\right]\\
&\leq c \left( \sum_{j=0}^{\varphi_{s_2}-1} \E\left[|N_{t_{j+1}}|^2-|N_{t_{j}}|^2 + |N_{s_2}|^2-|N_{t_{\varphi_{s_2}}}|^2\right] \right)\\
&= c\E[|N_{s_2}|^2],
\end{align*} 
thus $\sup_n \E\left[\left| \sum_{j=0}^{\varphi_{s_2}-1} \delta_j u \; \Delta_j N + (N_{s_2}-N_{\varphi_{s_2}}) \right|^2\right]\leq c \E[|N_{s_2}|^2]<\infty$. Applying Lebesgue's dominated convergence Theorem in \eqref{eq:martingale} we get
\begin{align*}
\E[P_{s_2} \vert \mathcal{F}_{s_1}]&=\lim_{n \to \infty}  \E\left[\sum_{j=0}^{\varphi_{s_2}-1} \delta_j u \; \Delta_j N + (N_{s_2}-N_{t_{\varphi_{s_2}}}) \vert \mathcal{F}_{s_1}\right]\\
&=\lim_{n \to \infty} \Bigg( \sum_{j=0}^{\varphi_{s_1}-1} \delta_j u \; \Delta_j N + \E[(\delta_{\varphi_{s_1}} u) \; \Delta_{\varphi_{s_1}} N\vert\mathcal{F}_{s_1}]\\ 
&\quad \quad \quad +\E\Bigg[\sum_{j=\varphi_{s_1}+1}^{\varphi_{s_2}-1} \delta_j u \; \Delta_j N + (N_{s_2}-N_{t_{\varphi_{s_2}}})\vert \mathcal{F}_{s_1}\Bigg]\Bigg)\\ 
&=P_{s_1} + \lim_{n \to \infty} \left((\delta_{\varphi_{s_1}} u) \; (N_{s_1}-N_{t_{\varphi_{s_1}}})\right)=P_{s_1}
\end{align*} 
where for the last equality we have used the fact that $u$ is bounded and the continuity property of the martingale $N$. Putting all the previous facts together, $P$ is a continuous martingale which by definition has finite variations so 
$$[ N, N ]_s=[ N, N ]_0=0, \quad \textrm{ for all } s \in [0,T], \; \P-a.s.$$
which entails that $N \equiv 0$.
\end{proof}
The proof of Theorem \ref{th:main} will be given thanks to an approximating procedure.  
More precisely, let $F:\real \to \real$ be a bounded Borelian function. For $\varepsilon$ in $(0,1)$ we set:
$$ \phi_\varepsilon(x):=\frac{\exp\left(-\sum_{i=1}^d \frac{x_i^2}{2 \varepsilon}\right)}{(2 \pi \varepsilon)^{d/2}}, \quad \forall x\in \real^d,$$
and
\begin{equation}
\label{eq:Fe}
F_\varepsilon(x):=(F \ast \phi_\varepsilon)(x) := \int_{\real} F(y) \phi_\varepsilon(x -y) dy, \quad \forall x\in \real^d.
\end{equation}
Hence $F_\varepsilon$ is a Borelian bounded function which is infinitely differentiable. 
This regularity will ensure the representation of $F_\e(M_T)$ without orthogonal component.
\begin{lemma}
\label{lemma:Markovprop}
For any $\varepsilon$ in $(0,1)$ let
$Y^\e:=\E[F_\e(M_T) \vert \mathcal{F}_\cdot]$. There exists a bounded deterministic function $u_\e:[0,T] \times \real^d \to \real$ in $C^{0,1}([0,T] \times \real^d)$ such that
$$ Y^\e_t=u_\e(t,M_t), \quad t\in [0,T], \;\P-a.s..$$
\end{lemma}

\begin{proof}
Using the Markov property of $M$ we directly have that $u_\e(t,x):=\E[F_\e(M_T^{t,x})]$ where we recall that $M_s^{t,x}:=x+M_s-M_t$ for every $s\in [t,T]$. 
To show that $u_\e$ is Borelian we will prove that it is continuous. Indeed, let $(s,t,x,y)$ in $[0,T]^2\times (\real^d)^2$. The fact that $F_\varepsilon$ is bounded and Lebesgue's dominated convergence Theorem yield
\begin{align*}
\lim_{(s,x)\to (t,y)} |u_\e(s,x)-u_\e(t,y)|&\leq \lim_{(s,x)\to (t,y)} \E[|F_\e(M_T^{s,x})-F_\e(M_T^{t,y})|]\\
&= \E[\lim_{(s,x)\to (t,y)} |F_\e(M_T^{s,x})-F_\e(M_T^{t,y})|]\\
&= \E[\lim_{(s,x)\to (t,y)} |F_\e(x+M_T-M_s)-F_\e(y+M_T-M_t)|]=0 
\end{align*}
since $F_\e$ and $M$ are continuous.
We now deal with the differentiability in space. To this end we prove that $F_\e$ is Lipschitz. 
Let $x$ in $\real^d$ and $i\in\{1,\ldots,d\}$. We have
\begin{align*}
|\partial_{x_i} F_\e(x)|&=|(F \ast \partial_{x_i} \varphi_\e)(x)| \\
&=\left| \int_{\real^d} F(y) \partial_{x_i} \Phi_\e(x-y) dy \right|\\
&=\frac{1}{\varepsilon} \left| \int_{\real^d} F(y) (x_i-y_i) \Phi_\e(x-y) dy \right|\\
&\leq \frac{c}{\varepsilon} \int_{\real^d} |x_i-y_i| \Phi_\e(x-y) dy \\
&= \frac{c}{\varepsilon} \int_{\real} |x_i-y_i| \exp\left(-\frac{(x_i-y_i)^2}{2\epsilon}\right)(2 \pi \e)^{-1/2} dy_i \\
&\leq \frac{c}{\varepsilon} \left| \int_{\real} |x_i-y_i|^2 \exp\left(-\frac{(x_i-y_i)^2}{2\epsilon}\right) dy_i \right|^{1/2}= \frac{c}{\varepsilon^{1/2}}
\end{align*}
which proves that $F_\e$ is Lipschitz continuous. For every element $x$ in $\real^d$, $1\leq i\leq d$, and $e_i:=(0,\ldots,0,1,0,\ldots,0)$ where the $1$ is at the $i$th component, we write $M^{i,t,x}$ for the $i$th component of $M^{t,x}$, $\partial_{x_i} F_\e$ for the partial derivative of $F_\e$ with respect to $x_i$ and $\nabla_x F_\e$ for the vector $(\partial_{x_1}F_\e,\ldots,\partial_{x_d}F_\e)$. We have 
\begin{align*}
&\lim_{\alpha \to 0} \frac{u(t,x+\alpha e_i)-u(t,x)}{\alpha}\\
&= \lim_{\alpha \to 0} \frac{\E[F_\e(M_T^{t,x+\alpha e_i})-F_\e(M_T^{t,x})]}{\alpha}\\
&= \lim_{\alpha \to 0} \frac{1}{\alpha} \E[\partial_{x_i} F_\e(M_T^{1,t,x},\ldots,M_T^{i-1,t,x},\bar{M},M_T^{i+1,t,x},M_T^{d,t,x})\underbrace{(M_T^{i,t,x+\alpha e_i}-M_T^{i,t,x})}_{=M_T^{i,t,x+\alpha e_i}-M_T^{i,t,x}=\alpha}]\\
&= \lim_{\alpha \to 0} \E[\partial_{x_i} F_\e(M_T^{1,t,x},\ldots,M_T^{i-1,t,x},\bar{M},M_T^{i+1,t,x},M_T^{d,t,x})]
\end{align*}
where $\bar{M}:=z+M_T^i-M_t^i$ with $z$ between $x+\alpha$ and $x$, and $M_T^{i,t,x}=x+M^i_T-M^i_t$ ($M^i$ being the $i$th component of $M$).  
Hence by continuity of $\partial_{x_i} F_\e$ and since $\partial_{x_i} F_\e$ is bounded, Lebesgue's dominated convergence Theorem implies that
$$ \lim_{\alpha \to 0} \frac{u(t,x+\alpha e_i)-u(t,x)}{\alpha} = \E[\partial_{x_i} F_\e(M_T^{t,x})].$$
Hence $\nabla_x u(t,x)=\E[\nabla_x F_\e(M_T^{t,x})]$. We finally prove that $\nabla_x u_\e$ is continuous. For this we prove that $\nabla_x F_\e$ is Lipschitz continuous. 
Let $1\leq i,j\leq d$. For $z$ in $\real^d$, it holds that  
\begin{align*}
\partial_{x_i x_j} F_\e(z)=(F \ast \partial_{x_i x_j} \Phi_\e)(z)&= \frac{1}{\e^2} \int_{\real^d} F(y) \Phi_\e(x-y) ((x_i-y_i) (x_j-y_j) - \varepsilon \textbf{1}_{i=j}) dy\leq \frac{c}{\e}.
\end{align*}
Fix $(t,x)$ in $[0,T]\times \real^d$. Let $s$ in $[0,T]$ and $y$ in $\real^d$ such that $|(s,y)-(t,x)| \leq 1$. We have
\begin{align*}
\lim_{(s,y)\to(t,x)} |\nabla_x u(t,x)-\nabla_x u(s,y)| &\leq \lim_{(s,y)\to(t,x)} \E[|\nabla_x F_\e(M_T^{t,x})-\nabla_x F_\e(M_T^{s,y})|]\\
&\leq c \lim_{(s,y)\to(t,x)} \E[|M_T^{t,x}-M_T^{s,y}|]\\
&= c \lim_{(s,y)\to(t,x)} \E[|x-y-(M_t-M_s)|]=0
\end{align*}
where for the last equality we have used Lebesgue's dominated convergence Theorem, since, $|M_t-M_s| \leq 2 \sup_{r\in [0,T]} |M_r|$ which is square integrable by Doob's inequality.
\end{proof}

We are now in position to prove Theorem \ref{th:main}.\\\\
\noindent
{\em{Proof of Theorem \ref{th:main}:}} The proof is done in two steps.\\\\
\noindent
\textbf{Step 1:} We first assume that $F$ is bounded. We define as in Lemma \ref{lemma:Markovprop}, for every $\varepsilon$ in $(0,1)$, the square integrable martingale $Y^\e:=\E[F_\e(M_T)\vert \mathcal{F}_\cdot]$ where $F_\e$ is defined by \eqref{eq:Fe}. The martingale $Y^\e$ admits the representation:
$Y^\e= Y_0^\e + \int_0^\cdot Z_s^\e dM_s + N^\e$
with $N^\e$ a square integrable c\`adl\`ag martingale such that $[M,N^\e]=\langle M, N^\e\rangle \equiv 0$. In addition, from Lemma \ref{lemma:Markovprop}, there exists a deterministic function $u_\e \in C^{0,1}([0,T]\times \real^d)$ such that $Y^\e=u_\e(\cdot,M)$. Hence $Y^\e$ fulfills the assumptions of Lemma \ref{lemma:rep1} and so $N^\e \equiv 0$ which means that $ Y^\e = Y_0^\e + \int_0^\cdot Z_s^\e dM_s$.
Let $\delta_\varepsilon$ be the square integrable martingale $\delta_\varepsilon:=Y^\varepsilon-Y$. Then Burkholder-Davis-Gundy's inequality and Doob's inequality imply that there exists a constant $\alpha>0$ independent of $\varepsilon$ such that
$$ \E\left[[N^\varepsilon-N]_T^2 \right] \leq \E\left[\int_0^T |(Z_s^\e-Z_s)q_s^\ast|^2 dC_s + [N^\varepsilon-N]_T^2 \right] \leq \alpha \E[|(F_\varepsilon-F)(M_T)|^2].$$
Now since $F_\varepsilon$ and $F$ are supposed to be bounded, we get using Lebesgue's dominated convergence Theorem that
$$ \lim_{\varepsilon \to 0} \E\left[[N^\varepsilon-N]_T^2 \right] \leq \alpha \E[|\lim_{\varepsilon \to 0} (F_\varepsilon-F)(M_T)|^2]=0$$
since for every $x$ in $\real^d$, it holds that $\lim_{\varepsilon \to 0} F_\varepsilon(x)=F(x)$.
Thus $N \equiv 0$ and 
$$ Y= Y_0 + \int_0^\cdot Z_s dM_s$$
(note that we get at the same time that $Z=\lim_{\varepsilon \to 0} Z^\e$ in $\mathcal{H}^2$). So the representation holds true without orthogonal part.\\\\
\noindent
\textbf{Step 2:} Now we consider a Borelian function $F$. Let $F^n(x):=\max(-n,\min(F(x),n))$ for every integer $n\geq 1$. Then the sequence $(|F^n-F|)_n$ is decreasing. Since every function $F^n$ is bounded, the result of Step 1 implies that each square integrable martingale $Y^n:=\E[F^n(M_T) \vert \mathcal{F}_\cdot]$ can be represented as $ Y^n=Y_0^n + \int_0^\cdot Z_s^n q_s^\ast dM_s $
with $Z^n$ in $\mathcal{H}^2$. Using once again Burkholder-Davis-Gundy and Doob inequalities applied to the martingale $Y^n-Y$, we get that      
$$ \E\left[[N]_T^2 \right] \leq \alpha \E[|(F^n-F)(M_T)|^2]$$
and $\lim_{n\to \infty} \E[|F^n-F(M_T)|^2]=0$ by monotone convergence Theorem which concludes the proof.
\hfill$\square$\\\par\medbreak

\begin{remark}
We have assumed that $M$ is a square-integrable martingale. Obviously, the results of this section hold true if one considers only a local martingale. Then all the proofs are unchanged up to a localization argument and naturally, the process $Z$ obtained in the decomposition only belongs to $\mathcal{H}^2_{loc}$.  
\end{remark}

\section{Existence of solutions of quadratic growth BSDEs under general filtration}
\label{section:existence}

We have proved that every random variable of the form $F(M_T)$ where $F$ is a Borelian map and $M$ a continuous Markov processes can be represented (up to a constant) as a stochastic integral of a predictable process against $M$, and we refer to such a property as a \textit{weak martingale representation property} for $M$. The usual martingale representation property is the basic ingredients to solve \textit{Backward Stochastic Differential Equations} which can be viewed as a non-linear version of martingale representation. We will prove in Section \ref{section:BSDEwithout} that the representation (without orthogonal component) we have obtained in the previous section can be extended to solving BSDEs driven by a continuous martingale without an orthogonal part in the continuous Markovian realm. But before that we need to fill some gap about the existence of BSDEs (with quadratic growth) driven by a continuous martingale with respect to a general filtration, that is a filtration which allows for discontinuous martingales. This program is realized in this Section.\\\\    
\noindent
More precisely, in this section, we will prove existence of a solution for quadratic growth BSDEs driven by a continuous martingale $M$ with respect to a general filtration $(\mathcal{F}_t)_{t\in[0,T]}$, and $M$ is not assumed (in this Section only) to be a Markov process anymore. We just assume that $M$ is a $d$-dimensional martingale with respect to a right-continuous complete filtration $(\mathcal{F}_t)_{t\in [0,T]}$. We use the notations and assumptions of Section \ref{section:preliminaries} (once again except the Markovian one). Note that the first results about existence and uniqueness of solutions for quadratic growth BSDEs with respect to a continuous martingale are given by Morlais in \cite{Morlais}, and rely on the fact that the filtration is continuous, whereas the Lipschitz case has been treated by El Karoui and Huang in \cite{ElKarouiHuang}. We will prove existence of a triple $(Y,Z,N)$ in $\mathcal{S}^2\times \mathcal{H}^2\times \mathcal{O}^2$ solution to:
\begin{equation}
\label{eq:BSDEexistence}
Y_t=\zeta -\int_t^T f(s,Y_s,Z_s) dC_s - \int_t^T Z_s dM_s - \int_t^T dN_s, \quad t\in [0,T].
\end{equation}
Here $\zeta$ is an $\mathcal{F}_T$-measurable bounded random variable and $f$ satisfies a quadratic growth condition of the form $|f(z,y,z)| \leq c(|y|+ \frac{\gamma}{2}|z|^2)$. Note that the terminal condition $\zeta$ and the driver $f$ are the data of the equation. The important feature here is that \textit{a priori} the martingale $N$ is only c\`adl\`ag and not continuous. In \cite{Morlais}, the proof of existence of a solution relies on the assumption that $N$ is continuous since the filtration is. The key argument consists in showing that if \eqref{eq:BSDEexistence} admits a solution, then the process $Y$ is bounded. This is proved in \cite[Lemma 3.1 (i)]{Morlais} (see also \cite[Theorem 2]{BriandHu1} when $M$ is a Brownian motion), by showing that a well-chosen process depending on $Y$ is a submartingale whose terminal value depends only on $\zeta$ and on some constants related to the growth condition of the driver $f$, which allows one to conclude that $Y$ is bounded (we refer to \cite[Lemma 3.1 (i)]{Morlais} for more details). However, the computations (which involve It\^o's-Tanaka's formula) leading to this submartingale break if the orthogonal component is not continuous any more, since in that case, extra terms introduced by the use of It\^o's-Tanaka's formula make the analysis intractable. One possible way to avoid this technicality is to assume that the driver $f$ does not depend on $Y$ (see \textit{e.g.} \cite[Lemma 3 (i)]{Morlaisjump1}). Once the boundedness of $Y$ is established one can combine a so-called Cole-Hopf transform with a monotone approximation procedure to deduce the existence of a solution. Very recently, Barrieu and El Karoui \cite{BarrieuElKaroui} have described a new method for providing solutions to quadratic growth BSDEs under continuous filtration, by simultaneously approximating the driver by a sequence of Lipschitz growth ones, and by controlling the norm of $Y$ uniformly. They then conclude using a pretty general monotone stability argument for so-called "quadratic semimartingales". The existence result for qgBSDEs is then in their paper a by-product of this monotone stability property for quadratic semimartingales.\\\\
\noindent
We will follow the main stream of the proof of \cite{BarrieuElKaroui} concerning the uniform control of the norm of the $Y$ process and the monotone approximation, and replace the stability argument by a result on compactness for general semimartingales due to Barlow and Protter in \cite{BarlowProtter}, and refine the estimates used in \cite{Morlais}. Note that in a sense, the compactness result established by Barrieu and El Karoui can be viewed as a deeper analysis of the result of Barlow and Protter in the continuous filtration framework. \\\\
\noindent
The main result of this section is the theorem below. 

\begin{theorem}
\label{th:existence}
Let $f:[0,T]\times \real\times \real^d \to \real$ be a continuous function in $(y,z)$ such that the following growth condition is satisfied:
$$|f(t,y,z)| \leq \eta_t+ b |y| + \frac{\gamma}{2} |z|^2, \quad \forall (s,y,z) \in [0,T]\times \real\times\real^d,$$
with $b$ a positive constant and, $\eta=(\eta_t)_{t\in [0,T]}$ a non-negative predictable process bounded by a positive constant $a$. Let $\zeta$ an $\mathcal{F}_T$-measurable bounded random variable. Then there exists a triple of processes $(Y,Z,N) \in \mathcal{S}^2\times \mathcal{H}^2\times \mathcal{O}^2$ solution to the BSDE:
$$ Y_t=\zeta + \int_t^T f(s,Y_s,Z_s q_s^\ast) dC_s - \int_t^T Z_s dM_s - \int_t^T dN_s, \quad t\in [0,T].$$
In addition $Y$ is bounded by a positive constant which only depends on $a$, $b$ and $\|\zeta\|_\infty$. Finally, the process $\int_0^\cdot Z_s dM_s + N$ is a BMO martingale.
\end{theorem}
 
\begin{proof}
The proof is done in several steps.\\\\
\noindent
\textbf{Step 1:}
Let $p\geq 1$ be an integer and set 
$$ q_p(s,y,z):=\frac{\gamma}{2} |z^2|\textbf{1}_{|z|\leq p} + (\gamma p |z|-\frac{\gamma}{2} p^2) \textbf{1}_{|z|>p} + b |y|+ \eta_s, \quad (s,y,z)\in [0,T]\times \real\times \real^d. $$
Since $q_p$ is uniformly Lipschitz in $(y,z)$, there exists (see \cite[Theorem 6.1]{ElKarouiHuang}) a (unique) triple $(U^p,V^p,W^p) \in \mathcal{S}^2\times \mathcal{H}^2\times \mathcal{O}^2$ solution to 
$$ U_t^p=\zeta + \int_t^T q_p(s,U_s^p,V_s^pq_s^\ast) dC_s - \int_t^T V_s^p dM_s - \int_t^T dW_s^p, \quad t\in [0,T].$$
The objective of Step 1 is to prove that:
\begin{itemize}
\item[(i)]$ U_t^p=\underset{\beta \in [-b,b], \; |\nu|\leq p}{\textrm{essusp}} \E_{\Q^\nu}\left[ \zeta e^{-\int_t^T \beta_r dC_r} - \int_t^T e^{-\int_t^s \beta_r dC_r} \eta_s dC_s - \frac{\gamma}{2} \int_0^t e^{-\int_t^s \beta_r dC_r} \nu_s q_s \nu_s^\ast dC_s \Big\vert \mathcal{F}_t\right]$
with $\frac{d\Q^\nu}{d\P}:=\mathcal{E}\left(\int_0^\cdot \gamma \nu_s dM_s \right)$, and $\beta$, $\nu$ two predictable continuous processes.
\item[(ii)] The sequence $(U^p)_p$ is increasing.
\item[(iii)] There exists a constant $\tilde{c}>0$ which only depends on $a$ (\textit{i.e.} the bound on $\eta$), $b$, on $\|\zeta\|_\infty$ such that $ |U^p| \leq \tilde{c}, \; \P-a.s., \; \forall p\geq 1.$ In particular $\tilde{c}$ is independent of $p$.  
\end{itemize}
The claim (i) is shown using duality arguments. Indeed, we have: 
$$q_p(s,y,z)=\sup_{\beta \in [-b,b], \; |\nu|\leq p} \{l_{\beta,\nu}(s,y,z)\}, \quad \forall (s,y,z) \in [0,T]\times \real\times \real^d$$
with $l_{\beta,\nu}(s,y,z):=-\beta_s y +\gamma z q_s \nu_s^\ast -\frac{\gamma}{2} \nu_s q_s \nu_s^\ast + \eta_s$.   
For processes $\beta, \nu$ as above, let $(Y^{\beta,\nu},Z^{\beta,\nu},N^{\beta,\nu}) \in \mathcal{S}^2\times \mathcal{H}^2\times \mathcal{O}^2$ be the solution to the Lipschitz BSDE:
$$ Y_t^{\beta,\nu}=\zeta + \int_t^T l_{\beta,\nu}(s,Y_s^{\beta,\nu},Z_s^{\beta,\nu}q_s^\ast) dC_s - \int_t^T Z_s^{\beta,\nu} dM_s - \int_t^T dN_s^{\beta,\nu}, \quad t\in [0,T].$$
Using comparison Theorem (Lemma \ref{lemma:comp}) and \cite[Proposition 3.2]{ElKarouiPengQuenez} (which clearly can be reproduced in the general continuous martingale BSDE setting) it holds that:
\begin{equation}
\label{eq:supthing}
U^p=\underset{\beta \in [-b,b], \; |\eta|\leq p}{\textrm{essusp}} Y^{\beta,\nu}.
\end{equation}
An application of integration by parts formula allows us to write for every processes $\beta \in [-b,b]$, $|\nu|\leq p$,
\begin{align*}
Y_t^{\beta,\nu}&= Y_T^{\beta,\nu} e^{-\int_t^T \beta_r dC_r} - \int_t^T e^{-\int_t^s \beta_r dC_r} (\frac{\gamma}{2} \nu_s q_s \nu_s^\ast -\eta_s) dC_s \\ 
&\quad - \int_t^T e^{-\int_t^s \beta_r dC_r} Z_s^{\beta,\nu} \underbrace{(dM_s - \gamma q_s^\ast q_s \nu_s^\ast dC_s)}_{=:dM_s^\nu} - \int_t^T e^{-\int_t^s \beta_r dC_r} dN_s^{\beta,\nu}
\end{align*}
Since $\nu$ is bounded, the process $L^\nu:=\mathcal{E}\left(\gamma \int_0^\cdot \nu_s dM_s\right)$ is a true martingale and we can define the probability measure $\mathbb{Q}^\nu$ as: $\frac{d\mathbb{Q}^\nu}{d\P}:=L^\nu$ under which $M^\nu$ and $N^{\beta,\nu}$ are true martingales. Hence we get:
$$ Y_t^{\beta,\nu} =\E_{\Q^\nu}\left[ \zeta e^{-\int_t^T \beta_r dC_r} - \int_t^T e^{-\int_t^s \beta_r dC_r} \eta_s dC_s - \frac{\gamma}{2} \int_0^t e^{-\int_t^s \beta_r dC_r} \nu_s q_s \nu_s^\ast dC_s \Big\vert \mathcal{F}_t\right]$$
which together with \eqref{eq:supthing} lead to
$$ U_t^p =\underset{\beta \in [-b,b], \; |\nu|\leq p}{\textrm{essusp}} \E_{\Q^\nu}\left[ \zeta e^{-\int_t^T \beta_r dC_r} - \int_t^T e^{-\int_t^s \beta_r dC_r} \eta_s dC_s - \frac{\gamma}{2} \int_0^t e^{-\int_t^s \beta_r dC_r} \nu_s q_s \nu_s^\ast dC_s \Big\vert \mathcal{F}_t\right].$$
Concerning (iii), note that for every processes $\beta$ and $\nu$ as above, it holds $\P$-a.s. that
\begin{align*}
&\zeta e^{-\int_t^T \beta_r dC_r} - \int_t^T e^{-\int_t^s \beta_r dC_r} \eta_s dC_s - \frac{\gamma}{2} \int_0^t e^{-\int_t^s \beta_r dC_r} \underbrace{\nu_s q_s \nu_s^\ast}_{\geq 0} dC_s\\
&\leq \zeta e^{-\int_t^T \beta_r dC_r} - \int_t^T e^{-\int_t^s \beta_r dC_r} \eta_s dC_s \leq  e^{b} (\|\zeta\|_\infty+a)
\end{align*}   
where we have used the fact that $q$ is a positive semidefinite matrix and that $|C|\leq \frac12$. This ends the proof of (iii). It remains to prove (ii) which is a direct consequence of comparison Theorem (Lemma \ref{lemma:comp}) and of the fact that by construction,  $|q_p| \leq |q_{p+1}|$ for every $p\geq 1$.\\\\
\noindent 
\textbf{Step 2:} Let $f:\Omega \times [0,T]\times \real\times \real^d \to \real$ be a continuous function in $(y,z)$ such that there exists $p\geq 1$ satisfying:
\begin{equation}
\label{eq:growth1}
-q_p(s,y,z) \leq f(t,y,z) \leq b |y| + \frac{\gamma}{2} |z|^2 + \eta_s, \quad \forall (s,y,z) \in [0,T]\times \real\times\real^d
\end{equation}
where $\eta$ is a predictable process bounded by a positive constant $a$.
Let $n\geq \lfloor \max(b,p)\rfloor$. We set:
$$ f_n(s,y,z):=\inf_{u,w}\{f(s,u,w) + n |y-u| + n|z-w|\}, \quad (s,y,z) \in [0,T]\times \real\times \real^d. $$
Let $(Y^n,Z^n,N^n) \in \mathcal{S}^2\times\mathcal{H}^2\times \mathcal{O}^2$ be the solution to the Lipschitz BSDE 
$$ Y_t^n=\zeta + \int_t^T f_n(s,Y_s^n,Z_s^n q_s^\ast) dC_s - \int_t^T Z_s^n dM_s - \int_t^T dN_s^n, \quad t\in [0,T].$$
Then,
\begin{itemize}
\item[(iv)] The sequence $(Y^n)_n$ is increasing and $|Y^n| \leq |U^n| \leq \tilde{c}, \; \P-a.s.$,
where $\tilde{c}>0$ is the same constant given in (iii) of Step 1 (so it is independent of $n$), and $(U^n,V^n,W^n)$ denotes the solution to the BSDE with terminal condition $\zeta$ and driver $q_n$ studied in the previous Step.
\item[(v)] There exists a triple of processes $(\hat{Y}^p,\hat{Z}^p,\hat{N}^p) \in \mathcal{S}^2\times \mathcal{H}^2\times \mathcal{O}^2$ solution to the BSDE:
$$ \hat{Y}_t^p=\zeta + \int_t^T f(s,\hat{Y}_s^p,\hat{Z}_s^p q_s^\ast) dC_s - \int_t^T \hat{Z}_s^p dM_s - \int_t^T d\hat{N}_s^p, \quad t\in [0,T]$$
and $|\hat{Y}^p| \leq \tilde{c}, \; \P-a.s.$ where $\tilde{c}>0$ is a positive constant so in particular it is independent of $p$. In addition, 
$$ \sup_{\tau \leq T} \E\left[ \int_\tau^T |\hat{Z}_s^p q_s^\ast|^2 dC_s + \int_\tau^T d[\hat{N}^p]_s \vert \mathcal{F}_\tau \right] \leq \tilde{\tilde{c}} $$
with $\tilde{\tilde{c}}$ another positive constant which also does not depend on $p$, and the supremum runs over all $\mathcal{F}$-stopping time smaller than $T$. In other words, $\int_0^\cdot \hat{Z}^p_s dM_s +\hat{N}^p$ is a BMO-martingale.
\end{itemize}
Since $n\geq p$, it holds that $|f_n| \leq \max(q_p,q_n) = q_n$. In addition by definition, $f_n \leq f_{n+1}$. Hence, comparison theorem in conjuction with the fact that $|f_n(s,y,z)| \leq q_n(s,y,z)$ and (iii) of Step 1, imply that $(Y^n)_n$ is increasing and $|Y^n| \leq |U^n| \leq \tilde{c}, \; \P-a.s.$. 
Now we turn out to the proof of claim (v). We have seen that the sequence $(Y^n)$ is increasing and bounded by a universal constant $\tilde{c}>0$. Hence the process $\hat{Y}^p:=\lim_{n\to\infty} Y^n$ is well defined and belongs to $\mathcal{S}^2$ since it is bounded. In addition, by construction $\lim_{n\to \infty} \|Y^n-\hat{Y}^p\|_{\mathcal{S}^2}=0$. By definition, $Y^n$ can be decomposed as $Y^n=Y_0^n + P^n + A^n$ with $P^n:=\int_0^\cdot Z^n_s dM_s + N_s^n$ and $A^n:=\int_0^\cdot f_n(s,Y_s^n,Z_s^n q_s^\ast) dC_s$. In order to prove the claim we will control each of these components and then use a compacity argument obtained by Barlow and Protter in \cite{BarlowProtter}. We start with the martingale part. We recall that $\tilde{c}>0$ is a constant independent on $n$ such that $|Y^n|<\tilde{c}, \P-a.s.$ for every $n\geq 1$. Define $\psi(x):=\frac{\exp(\gamma x)-1-\gamma x}{\gamma^2}$ and $\tilde{\psi}(x):=\psi(x+\tilde{c})$. For the estimate on the continuous part of $P^n$, we can adapt the argument of \cite[Lemma 3.1 (ii)]{Morlais}, and then we will complete the argument to get the estimate for its jump part. Let $\tau$ be a stopping time less or equal than $T$. Let $(\tau_k)_{k\geq 1}$ be a localizing sequence for the \textit{a priori} only local martingale $\int_0^\cdot \tilde{\psi}'(Y_{s-}) Z_s^n dM_s + \int_0^\cdot \tilde{\psi}'(Y_{s-}) dN^n_s$. Let $k \geq 1$. It\^o's formula yields
\begin{align*}
&\tilde{\psi}(Y_\tau^n) + \frac12 \int_\tau^{\tau_k} \tilde{\psi}''(Y_s^n) |Z_s^n q_s^\ast|^2 dC_s + \frac12 \int_\tau^{\tau_k} \tilde{\psi}''(Y_s^n) d[(N^n)^c]_s\\ 
&= \tilde{\psi}(Y_{\tau^k}^n) + \int_\tau^{\tau_k} \tilde{\psi}'(Y_s^n) f_n(s,Y_s^n,Z_s^n q_s^\ast) dC_s - \int_\tau^{\tau_k} \tilde{\psi}'(Y_s^n) Z_s^n dM_s - \int_\tau^{\tau_k} \tilde{\psi}'(Y_{s-}^n) dN^n_s\\ 
&\quad - \sum_{\tau<s\leq {\tau_k}} [\tilde{\psi}(Y_s^n)-\tilde{\psi}(Y_{s-}^n) - \tilde{\psi}'(Y_{s-}^n) \Delta_s Y^n]  
\end{align*}  
where $(N^n)^c$ denotes the continuous martingale part of $N^n$.
Now we compute the jump part of the formula above to get that is is non-negative. Indeed, denoting for simplicity $\tilde{Y}:=Y^n+\tilde{c}$, we get
\begin{align*}
&\sum_{\tau<s\leq {\tau_k}} [\tilde{\psi}(Y_s^n)-\tilde{\psi}(Y_{s-}^n) - \tilde{\psi}'(Y_{s-}^n) \Delta_s Y^n]\\
&=\sum_{\tau<s\leq {\tau_k}} \frac{e^{\gamma \tilde{Y}_{t-}}}{\gamma^2}[e^{\gamma \Delta_t \tilde{Y}}- 1 -\gamma \Delta_t \tilde{Y}] = \sum_{0<s\leq T} \frac{e^{\gamma \tilde{Y}_{t-}}}{\gamma^2}\Phi(\gamma \Delta_t \tilde{Y})
\end{align*}   
with $\Phi(x):=e^x-1-x$ which is non negative for every real number $x$. Hence coming back to the computations above we get that
\begin{align*}
&\frac12 \int_\tau^{\tau_k} \tilde{\psi}''(Y_s^n) |Z_s^n q_s^\ast|^2 dC_s + \frac12 \int_\tau^{\tau_k} \tilde{\psi}''(Y_s^n) d[(N^n)^c]_s\\ 
&\leq \tilde{\psi}(Y_T^n) + \int_\tau^{\tau_k} \tilde{\psi}'(Y_s^n) f_n(s,Y_s^n,Z_s^n q_s^\ast) dC_s - \int_\tau^{\tau_k} \tilde{\psi}'(Y_s^n) Z_s^n dM_s - \int_\tau^{\tau_k} \tilde{\psi}'(Y_{s-}^n) dN^n_s.
\end{align*}  
Now remark that $\psi'(x)\geq 0$ for $x\geq 0$, which implies in conjunction with the growth condition on $f_n$ (\textit{i.e.} $|f_n(t,y,z)| \leq b|y|+\frac{\gamma}{2}|z|^2$) and with the fact that $Y^n+\tilde{c} \geq 0$, that 
\begin{align*}
&\frac12 \int_\tau^{\tau_k} \tilde{\psi}''(Y_s^n) |Z_s^n q_s^\ast|^2 dC_s + \frac12 \int_\tau^{\tau_k} \tilde{\psi}''(Y_s^n) d[(N^n)^c]_s\\ 
&\leq \tilde{\psi}(Y_T^n) + \int_\tau^{\tau_k} \tilde{\psi}'(Y_s) (b|Y_s^n|+\frac{\gamma}{2}|Z_s^n q_s^\ast|^2 + \eta_s) dC_s - \int_\tau^{\tau_k} \tilde{\psi}'(Y_s) Z_s^n dM_s - \int_\tau^{\tau_k} \tilde{\psi}'(Y_{s-}) dN^n_s.
\end{align*}  
Hence 
\begin{align*}
&\frac12 \int_\tau^{\tau_k} |Z_s^n q_s^\ast|^2 (\tilde{\psi}''(Y_s^n)-\gamma \tilde{\psi}'(Y_s)) dC_s + \frac12 \int_\tau^{\tau_k} \tilde{\psi}''(Y_s^n) d[(N^n)^c]_s\\ 
&\leq \tilde{\psi}(Y_T^n) + \int_\tau^{\tau_k} (b \tilde{\psi}'(Y_s^n) |Y_s^n|+\eta_s) dC_s - \int_\tau^{\tau_k} \tilde{\psi}'(Y_s^n) Z_s^n dM_s - \int_0^{\tau_k} \tilde{\psi}'(Y_{s-}^n) dN^n_s.
\end{align*} 
In addition, by definition of $\tilde{\psi}$, we have that $\tilde{\psi}''-\gamma \tilde{\psi}'=1$, hence the previous inequality reads as
\begin{align*}
&\frac12 \int_\tau^{\tau_k} |Z_s^n q_s^\ast|^2 dC_s + \frac12 \int_\tau^{\tau_k} \tilde{\psi}''(Y_s) d[(N^n)^c]_s\\ 
&\leq \tilde{\psi}(Y_T^n) + \int_\tau^{\tau_k} (b \tilde{\psi}'(Y_s^n) |Y_s^n| + \eta_s) dC_s - \int_\tau^{\tau_k} \tilde{\psi}'(Y_s^n) Z_s^n dM_s - \int_\tau^{\tau_k} \tilde{\psi}'(Y_{s-}^n) dN^n_s.
\end{align*} 
Now taking conditional expectation with respect to $\mathcal{F}_\tau$ in the previous expression, using the fact that $Y^n$ is bounded by $\tilde{c}$ and noting that $\tilde{\psi}''\geq 0$, we finally get a constant $c_1>0$ depending only on the data of the equation and independent on $n$ such that $ \E\left[\int_\tau^{\tau_k} |Z_s^n q_s^\ast|^2 dC_s \Big\vert \mathcal{F}_\tau \right]\leq c_1.$
Taking the limit in the previous expression and using monotone convergence Theorem we get that $ \E\left[\int_\tau^T |Z_s^n q_s^\ast|^2 dC_s \Big\vert \mathcal{F}_\tau \right]\leq c_1$ 
and hence
\begin{equation}
\label{eq:estimatequadraZ}
\textrm{esssup}_{0\leq \tau \leq T} \E\left[\int_\tau^T |Z_s^n q_s^\ast|^2 dC_s \Big\vert \mathcal{F}_\tau \right]\leq c_1.
\end{equation}  
To complete our estimate for the term $P^n$ we now need an estimate on the quadratic variation of the orthogonal martingale $N^n$. This can be done as follows. Applying It\^o's formula to $|Y^n|^2$ we get that
\begin{align*}
&|Y_\tau^n|^2 + \int_\tau^T d[N^n]_s + \int_\tau^T |Z_s^n q_s^\ast|^2 dC_s \\
&=|Y_T^n|^2 + 2 \int_\tau^T Y_s^n f_n(s,Y_s^n,Z_s^nq_s^\ast) dC_s - 2 \int_\tau^T Y_s^n Z_s^n dM_s - 2 \int_\tau^T Y_{s-}^n dN_s^n 
\end{align*}
which leads to
\begin{align*}
&\E\left[\int_\tau^T d[N^n]_s \Big\vert \mathcal{F}_\tau \right] \leq \E[|Y_T^n|^2 \vert \mathcal{F}_\tau] + 2 \E\left[\int_\tau^T Y_s^n f_n(s,Y_s^n,Z_s^nq_s^\ast) dC_s \Big\vert \mathcal{F}_\tau\right]. 
\end{align*}
Using once again the growth condition on the driver $f_n$ and the fact that $Y^n$ is bounded by $\tilde{c}>0$ (which does not depend on $n$) we get that
$$ \E\left[\int_\tau^T d[N^n]_s \Big\vert \mathcal{F}_\tau \right] \leq c \left(1 + \E\left[\int_\tau^T |Z_s^n q_s^\ast|^2 dC_s \Big\vert \mathcal{F}_\tau\right]\right)$$
which leads to 
$$ \textrm{esssup}_{0\leq \tau \leq T} \E\left[\int_\tau^T d[N^n]_s \Big\vert \mathcal{F}_\tau \right] \leq c_2 $$
by \eqref{eq:estimatequadraZ}, where $c_2>0$ is a constant which only depends on $a$, $b$, $\gamma$ and $\|\zeta\|_\infty$.
Hence, we have proved that
$$ \textrm{esssup}_{0 \leq \tau \leq T} \E\left[ \int_\tau^T d[P^n]_s \Big\vert \mathcal{F}_\tau\right] = \textrm{esssup}_{0 \leq \tau \leq T} \E\left[ \int_\tau^T |q_s Z_s^n|^2 dC_s + \int_\tau^T d[N^n]_s \Big\vert \mathcal{F}_\tau \right] \leq c_3,$$
$c_3:=c_1+c_2$, so each $P^n$ is a BMO martingale whose BMO norm is uniformly bounded with respect to $n$ (indeed note that $\Delta P^n=\Delta N^n = \Delta Y^n$ and $Y^n$ is bounded by $\tilde{c}>0$, hence $|\Delta N^n| \leq 2 \tilde{c}$).
Now as a particular case of the previous result (taking $\tau=0$), using Burkholder-Davis-Gundy's inequality we have that:
$$ \E\left[\sup_{t\in [0,T]} |P^n_t|^2\right] \leq c \E[[P^n]_T] \leq \tilde{c}_3$$
where $\tilde{c}_3>0$ does not depend on $n$.
We now turn to a uniform control (with respect to $n$) of the the finite variation part $A^n$. Indeed, using the growth condition on $f_n$ and \eqref{eq:estimatequadraZ}, we have that 
\begin{align*}
\E[|A_T^n|]=\E\left[\left| \int_0^T f_n(s,Y_s^n,Z_s^n) dC_s \right|\right]&\leq \E\left[\int_0^T (b |Y_s^n|+\frac{\gamma}{2} |Z_s^n|^2 +\eta_s) dC_s \right]\\
&\leq c \left(1+\E\left[\int_0^T |Z_s^n q_s^\ast|^2 dC_s \right] \right) \leq c.
\end{align*}
Hence we have proved that there exists a constant $c>0$ independent on $n$ such that $\E[\sup_{t\in [0,T]} |P_t^n| + |A_T|]\leq c$ for every $n$. In addition since $\hat{Y}^p\overset{\mathcal{S}^2}{=} \lim_{n\to\infty} Y^n$, by \cite[Theorem 1]{BarlowProtter}, there exist a martingale $P$ and a finite variation process $A$ such that $\lim_{n\to \infty} \E[[P^n-P]_T^{1/2}]=0$, $\lim_{n\to\infty} \E[\sup_{t\in [0,T]} |A_t^n-A_t| =0$ and $\hat{Y}^p=P+A$. Since the driver $f$ is assumed to be continuous we get that, there exists a process $\hat{Z}^p$ in $\mathcal{H}^1$ and a martingale $\hat{N}^p$ with $\E[[N]_T^{1/2}]<\infty$ such that
$$ \hat{Y}^p_t=\zeta + \int_t^T f(s,\hat{Y}^p_s,\hat{Z}^p_s) dC_s - \int_t^T \hat{Z}^p_s dM_s - \int_t^T d\hat{N}^p_s, \quad t\in [0,T].$$
In addition, $|\hat{Y}^p| \leq \tilde{c}$ by construction. By reproducing the computations of this step with $(\hat{Y}^p,\hat{Z}^p,\hat{N}^p)$ and using the fact that $\hat{Y}^p$ is bounded we get that 
$$\textrm{esssup}_{0\leq \tau \leq T}\E\left[ \int_\tau^T |\hat{Z}^p_s q_s^\ast|^2 dC_s + \int_\tau^T d[\hat{N}^p]_s \Big\vert \mathcal{F}_\tau\right] \leq c_3$$
where $c_3$ is the very same constant obtained above which means that $\int_0^\cdot \hat{Z}^p_s dM_s + \hat{N}^p$ is a BMO martingale (since the jumps of $\hat{N}^p$ are bounded by $2 \tilde{c}$) and also that $(\hat{Z}^p,\hat{N}^p)\in \mathcal{H}^2 \times \mathcal{O}^2$. This is in fact a reformulation of the well-known link in the Brownian setting between a bounded terminal condition and the BMO property for the martingale part of the solution to a BSDE. 
\\\\
\noindent 
\textbf{Step 3:}
Let $f:[0,T]\times \real\times \real^d \to \real$ be a continuous function in $(y,z)$ such that the following growth condition is satisfied:
$$|f(t,y,z)| \leq \alpha |y| + \frac12 |z|^2 + \eta_s, \quad \forall (s,y,z) \in [0,T]\times \real\times\real^d.$$
Then there exists a triple of processes $(Y,Z,N) \in \mathcal{S}^2\times \mathcal{H}^2\times \mathcal{O}^2$ solution to the BSDE:
$$ Y_t=\zeta + \int_t^T f(s,Y_s,Z_s q_s^\ast) dC_s - \int_t^T Z_s dM_s - \int_t^T dN_s, \quad t\in [0,T].$$
In addition, there exists a constant $c_4>0$ which only depends on $\|\zeta\|_\infty$, $b$, and $\gamma$ such that $|Y| \leq c_4, \; \P-a.s.$, and the process $\int_0^\cdot Z_s dM_s + N$ is a BMO martingale.\\\\
\noindent
To prove this claim we use the usual decomposition $f=f^+-f^-$ for $f$. Now let $f^{-,p}(s,y,z):=\inf_{u,w}\{f^-(s,y,z) + p |y-u| + p |z-w|\}$. The function $f^{-,p}$ satisfies condition \eqref{eq:growth1} and we denote by $(\hat{Y}^p,\hat{Z}^p,\hat{N}^p)$ one solution to the BSDE with terminal condition $\zeta$ and driver $f^+-f^{-,p}$ obtained in Step 2. By comparison Theorem, the sequence $(\hat{Y}^p)_p$ is decreasing. In addition we have proved that each process $\hat{Y}^p$ is bounded by a universal constant $\tilde{c}>0$ which does not depend on $p$. As a consequence we can define $Y:=\lim_{p\to \infty} \hat{Y}^p$ and the convergence also holds in $\mathcal{S}^2$. 
Using again point (v) of Step 2, there exists a constant $\tilde{\tilde{c}}>0$ independent on $p$ such that:
$$ \textrm{esssup}_{0\leq \tau T}\E\left[ \int_\tau^T |\hat{Z}^p_s q_s^\ast|^2 dC_s + \int_\tau^T d[\hat{N}^p]_s \Big\vert \mathcal{F}_\tau \right] \leq \tilde{\tilde{c}}$$
which in conjuction with Burkholder-Davis-Gundy inequality implies that
$$ \E\left[\sup_{t\in [0,T]} \left|\int_0^t \hat{Z}^p_s dM_s + \hat{N}^p_t\right|\right] \leq c $$
where $c$ is a constant independent on $p$. Finally, we have that
$$ \E\left[ \left\vert \int_0^T f^{-,p}(s,\hat{Y}^p_s,\hat{Z}^p_s q_s^\ast) dC_s \right\vert \right] \leq c \left(1+\E\left[ \int_0^T |\hat{Z}^p_s q_s^\ast|^2 dC_s \right]\right) \leq c$$
where $c$ is a again a constant independent on $p$ which comes from the previous estimates. Hence using \cite[Theorem 1]{BarlowProtter}, there exists a process $Z$ in $\mathcal{H}^1$ and a martingale $N$ orthogonal to $M$ with $\E[[N]_T^{1/2}]<\infty$ such that
$$ Y_t=\zeta + \int_t^T f(s,Y_s,Z_s q_s^\ast) dC_s - \int_t^T Z_s dM_s - \int_t^T dN_s, \quad t\in [0,T].$$
Once again since $Y$ is bounded, we can reproduce the computations of Step 2 to prove that $(Z,N)$ belongs to $\mathcal{H}^2\times \mathcal{O}^2$ and even to prove that $\int_0^\cdot Z_s dM_s + N$ is a BMO martingale.     
\end{proof}

\begin{remark}
Here we did not address the question of uniqueness of solutions to keep the length of this paper within limits but we believe that under standard additional assumptions of the type \cite[Assumption (H2), p.~128]{Morlais} on the driver, the uniqueness could be obtained. This point is left for future research.   
\end{remark}

\section{Solving Markovian quadratic BSDEs without orthogonal component}
\label{section:BSDEwithout}

We now are in position to extend the weak martingale representation property obtained in Section \ref{section:wmr} to solving qgBSDEs driven by a $d$-dimensional strong Markov process $M$ which is also a martingale under a strong Markov filtration $(\mathcal{F}_{t\in [0,T]})$ (in the sense of \cite[(3.4)]{CinlarJacodProtterSharpe}) as described in Section \ref{section:preliminaries}. On the filtered probability space $(\Omega,\mathcal{F},(\mathcal{F}_t)_{t\in [0,T]},\P)$, we now consider in addition to the square integrable continuous martingale $M$, a stochastic process $X^{t,x,m}:=(X_s^{t,x,m})_{s\in [t,T]}$ defined as the unique strong solution of the following $n$-dimensional stochastic differential equation
\begin{equation}
\label{SDE}
X_s^{t,x,m}=x+\int_t^s \sigma(r,X_r^{t,x,m},M_r^{t,m}) dM_r + \int_t^s b(r,X_r^{t,x,m},M_r^{t,m}) dC_r, \quad s \in[t,T], \; t \in [0,T]
\end{equation}
where $\sigma: [0,T]\times \real^n \times \real^d \to \real^{n\times d}$, $b:[0,T]\times \real^n \times \real^d \to \real$ are deterministic functions of class $C^{0,1}([0,T]\times (\real^n \times \real^d))$ with locally Lipschitz partial derivatives in $x$ and $m$ uniformly in time, and such that there exists a positive constant $k$ satisfying 
\begin{align*}
&\vert \sigma(t,x_1,m_1)-\sigma(t,x_2,m_2) \vert + \vert b(t,x_1,m_1)-b(t,x_2,m_2) \vert\\
&\leq k (\vert x_1-x_2 \vert+\vert m_1-m_2 \vert), \quad \forall (t,x_1,x_2,m_1,m_2) \in [0,T] \times (\real^d)^2 \times (\real^n)^2,
\end{align*} 
where we use the notation $|\cdot|$ for both the Euclidian norm on $\real^n$ and on $\real^d$. 
Let us finally introduce the object of interest of this section that is the following BSDE coupled with the forward process $X^{t,x,m}$ as
\begin{align} 
\label{BSDE}
Y_s^{t,x,m}=&F(X_T^{t,x,m},M_T^{t,m})-\int_s^T Z_r^{t,x,m} dM_r + \int_s^T f(r,X_r^{t,x,m},M_r^{t,m},Y_r^{t,x,m},Z_r^{t,x,m}) dC_r\nonumber\\
&-\int_s^T dN_r^{t,x,m},
\end{align}
where $F:\real^n \times \real^d \to\real$ is a bounded deterministic function and the generator $f: [0,T] \times \real^n \times \real^d \times \real \times \real^d \to \real$ is assumed to be $\mathcal{B}([0,T])\otimes \mathcal{B}(\real^n \times \real^d \times \real \times \real^d)$-measurable ($\mathcal{B}(E)$ denoting the Borel $\sigma$-filed on an Euclidian space $E$), so that $f(r,x,m,y,z)$ is deterministic for non-random $(r,x,m,y,z)$ in $[0,T]\times\real^n \times \real^d \times \real \times \real^d$. In the main result of this section, we will make use of the assumptions below. For convenience, we set: $\Theta:=\real^d\times \real^n\times \real\times \real^{d}$.\\\\ 
\noindent
\textbf{(QG):} The driver $f:[0,T]\times \Theta \to \real$ is continuous in $(y,z)$ and there exist a nonnegative predictable process $\eta=(\eta_t)_{t\in [0,T]}$ bounded by a positive constant $a$, as well as positive constants $b$ and $\gamma >0 $ such that
$$ |f(t,x,m,y,z)| \leq \eta_t (1+b |y|) +\frac{\gamma}{2} |z|^2, \quad \forall (x,m,y,z) \in \Theta, \; \P\otimes dC_t-a.e..$$
In addition, for every $\beta>1$, $\int_0^T |f(t,0,0,0,0)| dC_t $ belongs to $L^\beta(d\P)$.\\\\   
\textbf{(D1):} The functions $F:\real^n \times \real^d \to \real$, $\partial_x F:\real^n \times \real^d \to \real^n$ and $\partial_m F:\real^n \times \real^d \to \real^d$ are globally Lipschitz.\\\\
\noindent
\textbf{(D2):} The driver $f:[0,T]\times \Theta$ is differentiable in $(x,m,y,z)$, there exist a positive constant $r$ such that $d\P\otimes dC_t-a.e.$ it holds:
\begin{align*}
&|\partial_{a} f(t,x,m,y,z)| \leq r (1 + |z|), \quad \forall (x,m,y,z) \in \Theta, \; a \in \{x,m,y,z\}, \textrm{ and }\\
&|\partial_{a} f(t,x_1,m_1,y_1,z_1)-\partial_{a} f(t,x_2,m_2,y_2,z_2)| \\
&\leq r (|q_t \theta_t| + |z_1| + |z_2|) (|x_1-x_2|+|m_1-m_2|+|y_1-y_2|+|z_1-z_2|), \quad a \in \{x,m,y\},
\end{align*}
for every $(x_i,m_i,y_i,z_i) \in \Theta, \; i=1,2$. Finally $\partial_z f$ is Lipschitz in $(x,m,y,z)$ uniformly in time. \\\\
\noindent
We have shown in Theorem \ref{th:existence} that under assumption (QG), there exists a triple $(Y^{t,x,m}, Z^{t,x,m}, N^{t,x,m})$ in $\mathcal{S}^2\times \mathcal{H}^2\times \mathcal{O}^2$ solution to the BSDE \eqref{BSDE}.

\subsection{Markov property of the solution}
\label{sub:Markov}

We start with an important fact about Markov processes which can be found in \cite[Theorem (8.11)]{CinlarJacodProtterSharpe} or in \cite[V. Theorem 35]{Protter}. 
\begin{prop}
The process $(X_s^{t,x,m},M_s^{t,m})_{s\in [t,T]}$ is a strong Markov process for the filtration $(\mathcal{F}_t)_{t\in [0,T]}$. If in addition $M$ is assumed to have independent increments then the stochastic process $(X_s^{t,x,m})_{s\in [t,T]}$ is a strong Markov process.
\end{prop} 
\noindent
We now want to prove that the Markov property of the pair $(X_s^{t,x,m},M_s^{t,m})_{s\in [t,T]}$ transfers to the solution of \eqref{BSDE}. The idea is usually to follow the proof of existence of a solution and to show \textit{on the way} that the solution process is given as a deterministic function of $(t,X,M)$. For the Lipschitz case, we can reproduce in our context the proof of \cite[Theorem 3.2]{IRR} without any changes since we use arguments which are valid for general martingales provided the process $C$ is continuous. However, in the quadratic case the proof of the Markovian feature of the solution given in \cite[Theorem 3.4]{IRR} was following the proof of existence given in \cite{Morlais}. However, with a possible discontinuous orthogonal component, this proof cannot be reproduced and we had to produce a new one. For that reason we will also have to give a new counterpart of \cite[Theorem 3.4]{IRR} in our setting. We first state the result in the Lipschitz case. 

\begin{prop}
\label{prop:MarkovpropLip}
Let $f:[0,T]\times \Theta \to \real$ such that there exists a constant $K>0$ such that
\begin{align*}
&|f(t,x_1,m_1,y_1,z_1)-f(t,x,m_2,y_2,z_2)| \\
&\leq K (|y_1-y_2|+|z_1-z_2|), \quad \forall (t,x_i,m_i,y_i,z_i) \in [0,T] \times \Theta, \; i=1,2
\end{align*}
and $\E\left[\int_0^T |f(t,0,0,0,0)|^2 dC_t \right]<\infty$.
Let $(Y^{t,x,m},Z^{t,x,m},N^{t,x,m})$ be the unique solution in $\mathcal{S}^2\times \mathcal{H}^2 \times \mathcal{O}^2$ of the BSDE
$$ Y_s^{t,x,m}=F(X_T^{t,x,m},M_T^{t,m})+\int_s^T f(r,X_r^{t,x,m},M_r^{t,m},Y_r^{t,x,m},Z_r^{t,x,m} q_r^\ast) dC_s - \int_s^T dN_r^{t,x,m}, \quad s\in [t,T].$$ 
Then, there exist two deterministic functions $u,v:[0,T] \times \real^n\times \real^d \to \real$, $\mathcal{B}([0,T]) \otimes \mathcal{B}_e(\real^n \times \real^d)$-measurable such that $(Y^{t,x,m},Z^{t,x,m})$ satisfy:
$$ Y_s^{t,x,m}=u(s,X_s^{t,x,m},M_s^{t,m}), \quad Z_s^{t,x,m}=v(s,X_s^{t,x,m},M_s^{t,m}), \quad s \in [t,T]$$
where $\mathcal{B}_e(\real^n\times \real^d)$ is the $\sigma$-field on $\real^n\times \real^d$ generated by the functions 
$$(x,m) \mapsto \E\L[\int_0^T \phi(s,X_s^{t,x,m},M_s^{t,m}) dC_s\R]$$ 
with $\phi:\Omega \times [0,T] \times \real^n\times \real^d \to \real$ a continuous bounded function.
\end{prop} 

\begin{proof}
The proof follows the line of the one of \cite[Proposition 3.2]{IRR}. Note that in \cite{IRR} the filtration was assumed to be continuous. However, for the particular result \cite[Proposition 3.2]{IRR} this assumption is not needed. Only the fact that $N$ is a martingale is important. So we can reproduce every argument of the proof without any modification. In order to keep the length of this paper within limits, we leave this point to the reader.  
\end{proof}

Now we deal with the quadratic case and give a counterpart to \cite[Theorem 3.4]{IRR}. 

\begin{theorem}
\label{th:Markovprop}
Assume that the driver $f$ satisfies (QG).
Let $(Y^{t,x,m},Z^{t,x,m},N^{t,x,m})$ be a solution in $\mathcal{S}^2\times \mathcal{H}^2 \times \mathcal{O}^2$ to the BSDE \eqref{BSDE} with driver $f$. Then there exists a deterministic function $u:[0,T] \times \real^n\times \real^d \to \real$, $\mathcal{B}([0,T]) \otimes \mathcal{B}_e(\real^n \times \real^d)$-measurable such that
$$ Y_s^{t,x,m}=u(s,X_s^{t,x,m},M_s^{t,m}), \quad s \in [t,T].$$
\end{theorem}

\begin{proof}
Since the construction of the solution to the BSDE \eqref{BSDE} differs from the one used in \cite{IRR} we produce here another proof which is in fact slightly simpler. We follow the steps and we use exactly the same notations of the proof of the existence result: Theorem \ref{th:existence}. We first assume that the driver $f$ satisfies \eqref{eq:growth1} like in Step 2 of the proof of Theorem \ref{th:existence}. To be more precise we know have that $f$ in Step 2 has a Markovian structure that is:
$$ f(s,y,z)(\omega)=f(s,X_s^{t,x,m}(\omega),M_s^{t,m}(\omega),y,z)$$
with the growth condition \eqref{eq:growth1} for some $p\geq 1$. Now the solution $\hat{Y}^{p,t,x,m}$ to this BSDE with this driver $f$ is given as: $\hat{Y}^{p,t,x,m}=\lim_{n\to\infty} Y^{n,t,x,m}$ where $Y^{n,t,x,m}$ denotes the solution to the BSDE with terminal condition $F(X_T^{t,x,m},M_T^{t,m})$ and driver $f_n$: the inf convolution of $f$ with the function $(u,w) \mapsto n|u|+n|w|$. By construction $f_n$ is a Lipschitz function thus from Proposition \ref{prop:MarkovpropLip}, there exists a maps $u^n, v^n:[0,T] \times \real^n\times \real^d \to \real$, $\mathcal{B}([0,T]) \otimes \mathcal{B}_e(\real^n \times \real^d)$-measurable such that 
$$Y^{n,t,x,m}=u^n(\cdot,X^{t,x,m},M^{t,m}), \; Z^{n,t,x,m}=v^n(\cdot,X^{t,x,m},M^{t,m}), \; \P-a.s..$$ 
In particular it holds that $u^n(t,x,m)=Y_t^{n,t,x,m}$. Letting 
$$u^p(t,x,m):=\limsup_{n\to\infty} u^n(t,x,m), \; v^p(t,x,m):=\limsup_{n\to\infty} v^n(t,x,m)$$
we obtain that $u^p, v^p$ are $\mathcal{B}([0,T]) \otimes \mathcal{B}_e(\real^n \times \real^d)$-measurable and 
$$\hat{Y}^{p,t,x,m}=u^p(\cdot,X^{t,x,m},M^{t,m}), \; \hat{Z}^{p,t,x,m}=v^p(\cdot,X^{t,x,m},M^{t,m}), \; \P-a.s.$$
(the convergence of $(v^n)_n$ to $v^p$ follows from the fact that $Z^p$ is the limit in $\mathcal{H}^2$ of the $(Z^n)_n$).
\noindent
We know relax the assumption on $f$ made before and we just assume that $(s,x,m,y,z)\mapsto f(s,x,m,y,z)$ satisfies (QG). Then once again as in Step 3 of the proof of Theorem \ref{th:existence}, the solution process $Y$ can be approximated by a sequence of processes $\hat{Y}^{p,t,x,m}$ has above (\textit{i.e.} $Y^{t,x,m}=\lim_{n\to\infty} \hat{Y}^{p,t,x,m}, \; \P-a.s.$, and $Z^{t,x,m}=\lim_{p\to\infty} \hat{Z}^{p,t,x,m}$ in $\mathcal{H}^2$). From what we have proved before, there exist $\mathcal{B}([0,T]) \otimes \mathcal{B}_e(\real^n \times \real^d)$-measurable maps $u^p, v^p$ such that in particular $\hat{Y}^{p,t,x,m}=u^p(\cdot,X^{t,x,m},M^{t,m})$ and $\hat{Z}^{p,t,x,m}=v^p(\cdot,X^{t,x,m},M^{t,m})$. Hence if we define $u(t,x,m):=\limsup_{p\to \infty} u^p(t,x,m)$ and $v(t,x,m):=\limsup_{p\to \infty} v^p(t,x,m)$ we get $u, v$ are $\mathcal{B}([0,T]) \otimes \mathcal{B}_e(\real^n \times \real^d)$ and $Y=u(t,X^{t,x,m},M^{t,m})$, $Z=v(t,X^{t,x,m},M^{t,m})$ hold true.           
\end{proof} 

\subsection{Regularity property of the solution}
\label{sub:Regul}

\noindent
With more assumptions on $F$ and $f$ we can even give regularity properties on the function $u$.

\begin{prop}
\label{prop:Diff} 
Assume that the assumptions (QG), (D1)-(D2) are in force. Then the function $u$ given in Theorem \ref{th:Markovprop} satisfies:
\begin{itemize}
\item[(i)] The map $ (x,m)\mapsto u(t,x,m) $ is of class $\mathcal{C}^{0,1}([0,T]\times (\real^n \times \real^d))$ for all $t\in [0,T]$. 
\item[(ii)] There exists two constants $\zeta_1$ and $\zeta_2$ depending only on $\|F\|_{\infty}$ and on the constants $a, b$ in Assumption (QG) such that $ \zeta_1 \leq u(t,x,m) \leq \zeta_2, \quad \forall (t,x,m) \in [0,T] \times \real^n \times \real^d.$
\end{itemize}
\end{prop}

\begin{proof}
From the Markovian representation of $Y^{t,x,m}$ we have that $u(t,x,m)=Y_t^{t,x,m}$. Hence part (ii) follows from the fact that the solution process $Y^{t,x,m}$ of \eqref{BSDE} is bounded by Theorem \ref{th:existence}. We now focus on part (i) of the theorem and more precisely on the continuity in $(t,x,m)$ of the map $u$. Let $(t_1,x_1,m_1), (t_2,x_2,m_2)$ in $[0,T]\times \real^n \times \real^d$ with $t_1<t_2$.
For simplicity, we use the following notations: $(X^i,M^i):=(X^{t_i,x_i,m_i},M^{t_i,m_i})$, $(Y^{i},Z^{i},N^i):=(Y^{t_i,x_i,m_i},Z^{t_i,x_i,m_i},N^{t_i,x_i,m_i})$, $i=1,2$, $\de X:=X^1-X^2$, $\de M:=M^1-M^2$, $\de Y:=Y^{1}-Y^{2}$, $\de Z:=Z^{1}-Z^{2}$, $\de N:=N^{1}-N^{2}$ and $\zeta:=F(X_T^{1},M_T^{1})-F(X_T^{2,},M_T^{2})$. Let $s\in [t_1,T]$. We have
\begin{equation}
\label{eq:diffBSDE}
\delta Y_s = \zeta -\int_s^T \de Z_r dM_r + \int_s^T g(r,\de Y_r,\de Z_r q_r^*) dC_r -\int_s^T d(\de N_r),
\end{equation} 
with $g(r,\de Y_r,\de Z_r q_r^*):=\de Z_r q_r^* A^Z_r+  \de Y_r A^Y_r + \de
M_r^* A^M_r + \de X_r^* A^X_r $ and
\begin{align*}
A^Z_r&:=\int_0^1 \nabla_{z}f(r,X_r^1,M_r^1,Y_r^1,Z_r^2 q_r^* +\alpha (Z_r^2-Z_r^1) q_r^*) d\alpha\\
A^Y_r&:=\int_0^1 \nabla_{y}f(r,X_r^1,M_r^1,Y_r^2+\alpha (Y_r^2-Y_r^1),Z_r^2q_r^*) d\alpha\\
A^M_r&:=\int_0^1 \nabla_{m}f(r,X_r^1,M_r^2+\alpha (M_r^2-M_r^1),Y_r^2,Z_r^2 q_r^*) d\alpha\\
A^X_r&:=\int_0^1 \nabla_{x}f(r,X_r^2+\alpha (X_r^2-X_r^1),M_r^2,Y_r^2,Z_r^2 q_r^*) d\alpha.
\end{align*}
In addition, since $f$ satisfies (QG) and (D1)-(D2), the condition (AP) (defined in Section \ref{section:Appendix}) is true for $g$ and so the apriori estimates of Proposition \ref{prop:aprioriortho} can be applied. More precisely, we obtain for any $p>1$ 
\begin{align}
\label{eq:cont1}
&|u(t_1,x_1,m_1)-u(t_2,x_2,m_2)|^{2p}\nonumber\\
&=\E[|Y_{t_1}^1-Y^2_{t_2}|^{2p}]\nonumber\\
&\leq c \left(\E[|Y_{t_1}^1-Y_{t_2}^1|^{2p}] +\E\left[|Y^1_{t_2}-Y^2_{t_2}|^{2p}\right]\right)\nonumber\\
&\leq c \left(\E[|Y_{t_1}^1-Y_{t_2}^1|^{2p}] +\E\left[|\zeta|^{2pq}+\left(\int_{t_2}^T |\delta M_r^\ast A^M_r + \delta X_r^\ast A^X_r| dC_r\right)^{2pq}\right]^{1/q}\right)
\end{align}
where $q>1$ comes from the a priori estimates of Proposition \ref{prop:aprioriortho}. Lebesgue's dominated convergence theorem leads to $\lim_{t_1 \to t_2, t_1<t_2} \E[|Y_{t_1}^1-Y_{t_2}^1|^{2p}]=\E[|\Delta_{t_1}Y^1|^{2p}]=\E[|\Delta_{t_1}N^1|^{2p}]=0$ since $N^1$ is a martingale (so its jumps does not occur at predictable times). In a similar way we have that $\delta X$ satisfies an SDE on $[t_2,T]$ with initial value $x_1-x_2+X_{t_2}^1-X_{t_1}^1$ and thus classical a priori estimates for SDEs (see \textit{e.g.} \cite[Lemma V.1]{Protter}) lead to (for any $k\geq 1$) 
\begin{equation}
\label{eq:contfor1}
\E[\sup_{s\in [t_2,T]} (|\delta X_s|^{2k} + |\delta M_s|^{2k})] \leq c (|x_1-x_2|^2+|m_1-m_2|^2+\E[|X_{t_2}^1-X_{t_1}^1|^2]+\E[|M_{t_2}-M_{t_1}|^2])^k. 
\end{equation} 
Coming back to \eqref{eq:cont1} and starting with the integral term we get:
\begin{align*}
&\E\left[\left(\int_{t_2}^T |\delta M_r^\ast A^M_r + \delta X_r^\ast A^X_r| dC_r\right)^{2pq}\right]\\
&\leq c \E\left[\left(\int_{t_2}^T (|\delta M_r^\ast|+|\delta X_r^\ast|)^2 dC_r \int_{t_2}^T |Z_r^2 q_r^\ast|^2 dC_r\right)^{pq}\right]\\
&\leq c \E\left[\left(\int_{t_2}^T (|\delta M_r^\ast|+|\delta X_r^\ast|)^2 dC_r\right)^{2pq}\right]^{1/2} \E\left[\left(\int_{t_2}^T |Z_r^2 q_r^\ast|^2 dC_r\right)^{2pq}\right]^{1/2}\\
\end{align*}
where we have used the fact that $|\nabla_x f(s,x,m,y,z) + \nabla_m f(s,x,m,y,z)|\leq c (1+|z|) $ by assumption $(D2)$, and two times Cauchy-Schwarz inequality. Since $\tilde{M}:=\int_0^\cdot Z_r^2 dM_r$ is a BMO martingale it belongs to $\mathcal{H}^k$ for every real number $k\geq 1$ (\textit{i.e.}, $\E[[\tilde{M}]_T^{k/2}]^{1/k}$), hence the term $\E\left[\left(\int_{t_2}^T |Z_r^2 q_r^\ast|^2 dC_r\right)^{2pq}\right]^{1/2}<\infty$. As a consequence, the previous computations together with the estimates \eqref{eq:contfor1} yield
\begin{align*}
&\E\left[\left(\int_{t_2}^T |\delta M_r^\ast A^M_r + \delta X_r^\ast A^X_r| dC_r\right)^{2pq}\right]\leq c \E\left[\left(\int_{t_2}^T (|\delta M_r^\ast|+|\delta X_r^\ast|)^2 dC_r\right)^{2pq}\right]^{1/2}\\
&\leq c (|x_1-x_2|^2+|m_1-m_2|^2+\E[|X_{t_2}^1-X_{t_1}^1|^2]+\E[|M_{t_2}-M_{t_1}|^2])^{2pq}.
\end{align*}  
Lebesgue dominated convergence Theorem and the continuity of $X^1$ and $M$ imply that 
$$ \lim_{(t_1,x_1,m_1)\to(t_2,x_2,m_2) t_1<t_2} \E\left[\left(\int_{t_2}^T |\delta M_r^\ast A^M_r + \delta X_r^\ast A^X_r| dC_r\right)^{2pq}\right]=0.$$
Now it remains to deal with the term $\E[|\zeta|^{2pq}]$ in \eqref{eq:cont1}. To this end, we apply the Lipschitz assumption on $F$ to get
$$ \E[|\zeta|^{2pq}] \leq c \E[|X^1_T-X_T^2|^{2pq} + |M^1_T-M_T^2|^{2pq}] $$
and we conclude with \eqref{eq:contfor1}. Putting together these results we have proved that:
$$ \lim_{(t_1,x_1,m_1)\to(t_2,x_2,m_2) t_1<t_2} |u(t_1,x_1,m_1)-u(t_2,x_2,m_2)|^{2p}=0.$$ 
Obviously the case $t_1>t_2$ can be treated similarly.\\\\
We have proved that the function $u$ is continuous. In addition we know that $Y^{t,x,m}=u(t,X^{t,x,m},M^{t,m})$, hence $Y$ is continuous. As a consequence the orthogonal component $N$ is also continuous. So our situation now perfectly matches the setting of \cite{IRR}, we can directly apply their result to get that for every $t \in [0,T]$, the map $(x,m) \mapsto u(t,x,m)$ is continuously differentiable (we refer to \cite[Propositions 4.5 and 4.7]{IRR}). Note that incidently some typos can be found in the proof of \cite[Proposition 4.7]{IRR} where a bound on quantities of the form $\E[\sup_{s\in [t_2,T]} (|X_s^{t_1,x_1,m_1}-X_s^{t_2,x_2,m_2}|^{2k} + |M_s^{t_1,m_1}-M_s^{t_2,m_2}|^{2k})]$ are given only in terms of the difference between $m_1-m_2$ whereas they should be given as in \eqref{eq:contfor1}. Anyway this quantity tends to zero as $(t_1,x_1,m_1)$ goes to $(t_2,x_2,m_2)$ which is enough for our purpose. Note also that as in \cite[Proposition 4.5]{IRR} we can prove that $(\partial_x Y^{t,x,m},\partial_x Z^{t,x,m},\partial_x N^{t,x,m})$ satisfies a BSDE.   
\end{proof}

\subsection{Representation without orthogonal component}
\label{section:without}

We can now state and prove the main result of this section.

\begin{theorem}
\label{theorem:main}
Assume that $f$ satisfies the assumptions (QG) and (D2). Let $F$ be bounded Borelian map.
Then $N^{t,x,m}$ in \eqref{BSDE} is equal to zero and equation \eqref{BSDE} becomes
$$ Y_s^{t,x,m}=F(X_T^{t,x,m})-\int_s^T Z_r^{t,x,m} dM_r + \int_s^T f(r,Y_r^{t,x,m},Z_r^{t,x,m} q_r^\ast) dC_r.$$
\end{theorem}

\begin{proof}
We proceed again by approximation. For every $\varepsilon$ in $(0,1)$ we define $F_\e:=F \ast \Phi_\e$ where $\Phi_\e$ is defined by \eqref{eq:Fe}. Hence, $F_\e$ satisfies assumption (D1) (we refer to the proof of Lemma \ref{lemma:Markovprop} where we have proved that for every bounded function $F$, the associated $F_\e$ functions are Lipschitz with Lipschitz derivative). Now we set $(Y^\e,Z^\e,N^\e)$ the solution of the BSDE \eqref{BSDE} where $F$ is replaced by $F_\e$, meaning:
$$ Y_t^\e=F^\e(X_T^{t,x,m},M_T^{t,m}) + \int_t^T f(s,X_s^{t,x,m},M_s^{t,m},Y_s^\e,Z_s^\e q_s^\ast) dC_s -\int_t^T Z_s^\e dM_s -\int_t^T dN_s^\e.$$
By Theorem \ref{th:Markovprop} and Proposition \ref{prop:Diff}, there exists a deterministic function $u_\e:0,T] \times \real^n\times \real^d \to \real$ in $C^{0,1}([0,T]\times (\real^n \times \real^d))$ such that $ Y=u_\e(\cdot,X,M).$ Hence by Lemma \ref{lemma:rep1}, $N^\e\equiv 0$. Let $\de Y:=Y^{t,x,m}-Y^\e$, $\de Z:=Z^{t,x,m}-Z^\e$, $\zeta^\e:=F(X_T^{t,x,m},M_T^{t,m})-F^\e(X_T^{t,x,m},M_T^{t,m})$ we have that $(\de Y, \de Z, N^{t,x,m})$ is solution to 
\begin{align*}
\de Y_s= \zeta^\e +\int_s^T g(r,\de Y_r, \de Z_r q_r^\ast, N_r) dC_r -\int_t^T \de Z_r dM_r - \int_t^T dN_r^{t,x,m}, \quad s\in [t,T]
\end{align*}  
with $g(r,\de Y_r,\de Z_r q_r^\ast):=\de Z_r q_r^\ast A_r^Z+ A_r^Y \de Y_r$ and $A_r^Y$, $A_r^Z$ defined in a similar way than in the proof of Proposition \ref{prop:Diff}.
By (D2), the generator $g$ satisfies (AP) of Section \ref{section:Appendix} and so a priori estimates of Proposition \ref{prop:aprioriortho} allows us to write for every $p>1$
$$\E\L[\sup_{s \in [t,T]} \vert \de Y_s\vert^{2 p}\R] + \E\left[ \left(\int_t^T \vert \de Z_r q_r^\ast \vert^2 dC_r \right)^p\right] + \E\left[[N^{t,x,m}]_T^{p}\right]\leq c \E\left[ \vert \zeta^\e \vert^{2 pq} \right]^{\frac{1}{q}}$$
where $q\geq 1$ is given by Proposition \ref{prop:aprioriortho}. Since $F$ is bounded, $F^\e$ is also bounded and Lebesgue dominated convergence Theorem leads to $\lim_{\varepsilon \to 0} \E[\vert \zeta^\e \vert^{2 pq}]=0$. Hence $N^{t,x,m}\equiv 0$. 
\end{proof}

\begin{remark}
We restrict our result to the case where the function $F$ in terminal condition is bounded because existence results for quadratic growth BSDEs are essentially valid under this hypothesis. To be more precise it is still possible to prove existence of solutions of BSDEs whose terminal condition admits finite exponential moments and we refer to \cite{BriandHu1}. However, this result has been proved in the Brownian framework and differentiability results have not been given up to our knowledge. Finally, if the driver $f$ is uniformly Lipschitz in $(y,z)$ then Theorem \ref{theorem:main} is valid for any Borelian map $F$ such that $F(X_T^{t,x,m},M_T^{t,m})$ is square integrable (the proof is similar to the Step 2 of the proof of Theorem \ref{th:main}).   
\end{remark}

\begin{remark}
Note that such an exact representation allows for numerical schemes for the solution of qgBSDEs driven by a general continuous Markov process. Indeed, usual such schemes are attainable only in a Markovian framework and it would be very difficult to simulate an orthogonal component whose structure is completely unknown. Our result then opens the way to the study of schemes for qgBSDEs driven by continuous Markov processes other than the standard Brownian motion.             
\end{remark}

\begin{remark}
Another application of Theorem \ref{theorem:main} is that we can obtain so-called representation formulas for the solutions of BSDEs \eqref{BSDE} with its application to cross-hedging in Finance without the assumption called (MRP) in \cite{IRR} which was crucial in \cite{IRR}.               
\end{remark}

\section{Appendix: a priori estimates and comparison theorem}
\label{section:Appendix}

We collect in this Appendix first a comparison theorem for Lipschitz growth BSDEs and a priori estimates. Here the main point is that the stochastic process $M$ is a continuous $d$-dimensional martingale with respect to a right-continuous complete filtration $(\mathcal{F}_t)_{t\in [0,T]}$. Since the later is not assumed to be continuous, the orthogonal component $N$ solution to the BSDE as above is just a square integrable c\`adl\`ag martingale. In this Appendix, $M$ is not assumed to be a Markov process. We start with a comparison theorem for the Lipschitz case whose proof basically follows the usual setting (see \textit{e.g.} \cite[Theorem 2.2]{ElKarouiPengQuenez} or \cite[Chapter 1 Theorem 6.1]{MaYong}). We nevertheless give a proof for making the paper self-contained.

\begin{lemma}
\label{lemma:comp}
Let $\zeta_1, \; \zeta_2$ in $L^2(\mathcal{F}_T)$, $f_1, f_2:[0,T] \times \Omega\times \real\times\real^d \to real$ two functions and $K_1, \;K_2 >0$ two positive constants satisfying
$$ |f^i(t,y,z)-f^i(t,y',z')| \leq K_i, \quad \forall (t,y,y',z,z')\in [0,T]\times \real^2\times (\real^d)^2,. \; i=1,2.$$
Let $(Y^i,Z^i,N^i)\in \mathcal{S}^2\times \mathcal{H}^2 \times \mathcal{O}^2$ be the solution to the BSDE:
$$ Y_t^i=\zeta_i + \int_t^T f_i(s,Y_s^i,Z_s^i q_s^\ast) dC_s - \int_t^T Z_s^i dM_s -\int_t^T dN_s^i, \quad t\in [0,T], \quad i=1,2. $$
If $\zeta_1\geq \zeta_2, \; \P-a.s.$ and $f_1(t,Y_t^2,Z_t^2 q_t^\ast) \geq f_2(t,Y_t^2,Z_t^2 q_t^\ast), \; d\P\otimes dC_t-a.s.$, then $Y^1 \geq Y^2, \; d\P\otimes dC_t-a.s.$. 
\end{lemma}               
\begin{proof}
Let $\delta Y:=Y^1-Y^2$, $\delta Z:=Z^1-Z^2$, $\delta N:=N^1-N^2$, $\delta \zeta:=\zeta^1-\zeta^2$ and $\delta f:=f^1(\cdot,Y^2,Z^2 q^\ast)-f^2(\cdot,Y^2,Z^2 q^\ast)$. We set: 
\begin{align*}
J_s&=\begin{cases}
     \frac{f^1(s,Y_s^1,Z_s^2 q_s^*)-f^1(s,Y_s^2,Z_s^2 q_s^*)}{Y_s}, & \mbox{if } \de Y_s \neq 0,  \\ 0, & \mbox{otherwise, and}
    \end{cases}
\\
H_s&=\begin{cases}
     \frac{f^{1}(s,Y_s^1,Z_s^1 q_s^*)-f^1(s,Y_s^1,q_s^2 q_s^\ast)}{|Z_s q_s^\ast|^2} \de Z_s, & \mbox{if } |\de Z_s q_s^\ast|^2 \neq 0,  \\ 0, & \mbox{otherwise.}
    \end{cases}
\end{align*}
The triple $(\delta Y, \de Z, \de N)$ is solution to the following linear BSDE 
$$ \de Y_t =\de \zeta +\int_t^T (J_s \de Y_s + \de Z_s q_s^\ast (H_s q_s^\ast)^\ast + \de f_s) dC_s -\int_t^T \de Z_s dM_s -\int_t^T \de N_s, \quad t\in [0,T].$$
This expression rewrites as
$$ \de Y_t =\de \zeta +\int_t^T (J_s \de Y_s + \de f_s) dC_s -\int_t^T \de Z_s \underbrace{(dM_s-q_s^\ast q_s H_s dC_s)}_{=:dM^H}  -\int_t^T \de N_s, \quad t\in [0,T].$$
Letting $L^H:=\mathcal{E}(\int_0^\cdot H_s dM_s)$. Since $|H q|$ is bounded, $L^H$ is a true martingale, so are $M^H$ and $N$ under $\Q$ with $\frac{d\Q}{d\P}:=L^H_T$. 
Let $\Gamma(t):=e^{\int_0^t J_s dC_s}$. Using integration by parts formula and taking conditional expectation with respect to $\Q$ in the expression above yield $\Gamma(t) \de Y_t =  \E^{\Q}\left[\Gamma_T \de Y_T +\int_t^T \Gamma_s \de f_s dC_s \Big\vert \mathcal{F}_t \right], \; \forall t\in [0,T]$
which concludes the proof. 
\end{proof}

We turn out to apriori estimates for stochastic Lipschitz BSDEs of the form \eqref{eq:BSDEexistence}
where the driver $f$ satisfies the assumption below:
\begin{itemize}
\item[\textbf{(AP):}] There exist a $\real^{1 \times d}$-valued predictable process $K$ and a constant $\alpha \in (0,1)$ such that $K \cdot M$ is a BMO martingale satisfying $d\P\otimes dC$-a.e.
$$(y_1-y_2)(f(s,y_1,z)-f(s,y_2,z)) \leq \vert q_s K_s^* \vert^{2\alpha}|y_1-y_2|^2, \forall (y_i,z) \in \real\times\real^d, \; i=1,2,$$
$$\textrm{ and }|f(s,y,z_1)-f(s,y,z_2)| \leq \vert q_s K_s^* \vert |z_1-z_2|, \forall (y,z_i) \in \real\times\real^d, \; i=1,2.$$
\end{itemize}

\begin{prop}
\label{prop:aprioriortho}
Let $(Y,Z,N)$ be a solution to \eqref{eq:BSDEexistence} with driver satisfying condition (AP) and $\zeta$ is a bounded random variable (so here $N$ is a priori only c\`adl\`ag).
We assume that for every $\beta \geq 1$ we have $\int_0^T \vert f(s,0,0) \vert dC_s \in L^\beta(\P)$.
Let $p>1$, then there exist constants $q \in (1, \infty)$, $c>0$ depending only on $T$, $p$ and on the BMO-norm of $K \cdot M$ such that
\begin{align*}
&\E\L[\sup_{t \in [0,T]} \vert Y_t\vert^{2 p}\R] + \E\left[ \left(\int_0^T \vert q_s Z_s^* \vert^2 dC_s \right)^p\right] + \E\left[ [ N]_T^{p}\right]
\\
&\leq c \E\left[\vert \zeta \vert^{2 pq} + \left( \int_0^T \vert
f(s,0,0) \vert dC_s  \right)^{2pq}\right]^{\frac{1}{q}}.
\end{align*}
\end{prop}

\begin{Proof}
Here we sketch the main steps of the proof which basically follows the same lines that the proof of \cite[Lemma A.1]{IRR} and we will only indicate the main changes in the proof. We proceed in several steps.

\noindent In a first step we exploit properties of BMO-martingales.
Let
\begin{align*}
J_s&=\begin{cases}
     \frac{f(s,Y_s,Z_s q_s^*)-f(s,0,Z_s q_s^*)}{Y_s}, & \mbox{if } Y_s \neq 0,  \\ 0, & \mbox{otherwise, and}
    \end{cases}
\\
H_s&=\begin{cases}
     \frac{f(s,0,Z_s q_s^*)-f(s,0,0)}{|q_s Z_s^*|^2}Z_s, & \mbox{if } |q_s Z_s^*|^2 \neq 0,  \\ 0, & \mbox{otherwise.}
    \end{cases}
\end{align*}
Then BSDE \eqref{eq:BSDEexistence} has the form
\begin{equation*}
 Y_t=\zeta-\int_t^T Z_s dM_s +\int_t^T \left(J_s Y_s +(q_s H_s^*)(q_sZ_s^*)^*+
 f(s,0,0)\right)dC_s -\int_t^T dN_s,\quad t\in[0,T].
\end{equation*}
Due to (AP) we have $|q H^*| \leq |q K^*|$ and it follows that $H \cdot M$ is a BMO($\P$) martingale.
Furthermore we know from \cite[Theorem 3.1]{Kazamaki} that there exists a $\hat q >1$ such that the reverse H\"older inequality holds, i.e. there exists a constant $c>0$ such that $\mathcal E(H \cdot M)_t^{- \hat q} \E \left[ \mathcal E(H \cdot M)_T^{\hat q} | \mathcal F_t \right] \leq c$.
By \cite[Theorem 2.3]{Kazamaki} the measure $\Q$ defined by $d\Q=\mathcal E(H \cdot M)_T d\P$ is a probability measure.
Girsanov's theorem implies that $\Lambda:=\int_0^\cdot Z_s dM_s - \int_0^\cdot(q_s H_s^*)(q_sZ_s^*)^* dC_s + N$
is a local $\Q$-martingale. Note that the difference with the same expression in proof of \cite[Lemma A.1]{IRR} is that $N$ is in the definition of $\Lambda$ but not in the definition of $\Q$. To be more precise, $\Q$ takes into account the quadratic growth of the driver which is only in $Z$ so there is no need to introduce $N$ in the measure changes. However, due to the presence of $N$ now in the equation, we need to introduce $N$ inside the quantity $\Lambda$.
As in the initial proof, there exists an increasing sequence of stopping times $(\tau^n)_{n\in\mathbb{N}}$ converging to $T$ such that $\Lambda_{\cdot \wedge \tau^n}$ is a $\Q$-martingale for any $n\in\mathbb{N}$.
Letting $e_t=\exp(2\int_0^t |q_s K_s^*|^{2 \alpha} dC_s)$, $t \in [0,T]$, with It\^o's formula applied to $e_t Y_t^2$ we have
\begin{align*}
d (e_t |Y_t|^2) &= 2 |q_t K_t^*|^{2 \alpha}  e_{t} Y_{t}^2 dC_t + 2 e_t Y_{t-} dY_t + e_t |q_t Z_t^*|^2 dC_t + e_t d[ N^c ]_t + e_t d[ N ]_t \\
& = 2 e_t |Y_t|^2 (|q_tK_t^*|^{2 \alpha}-J_t) dC_t + 2 e_t Y_{t-} d \Lambda_t - 2 e_t Y_t f(t,0,0) dC_t+ e_t |q_t Z_t^*|^2 dC_t + e_t d[ N ]_t,
\end{align*}
since $C$ is a continuous process.
With the inequality $J_t \leq |q_t K_t^*|^{2 \alpha}$, $t \in [0,T]$, which follows from assumption (AP) we know for $t\in[0,\tau^n]$
\begin{align*}
 e_t |Y_t|^2
\leq
& e_{\tau^n} |Y_{\tau^n}|^2
- \int_t^{\tau^n} 2 e_s Y_s d \Lambda_s
+ \int_t^{\tau^n} 2 e_s Y_s f(s,0,0) dC_s
- \int_t^{\tau^n} e_s |q_s Z_s^*|^2 dC_s - \int_t^{\tau^n} e_s d[ N ]_s.
\end{align*}
Note that $e_t \geq 1$ for all $t \in [0,T]$ and hence
\begin{align*}
e_t |Y_t|^2 + \int_t^{\tau^n} |q_s Z_s^*|^2 dC_s + \int_t^{\tau^n} d[ N ]_s
\leq
& e_{\tau^n} |Y_{\tau^n}|^2
- \int_t^{\tau^n} 2 e_s Y_s d \Lambda_s
+ \int_t^{\tau^n} 2 e_s Y_s f(s,0,0) dC_s.
\end{align*}
\noindent
In a second step we provide an estimate for $Y$. 
We take the conditional expectation under the new measure $\Q$ in the previous inequality.
As in the proof of \cite[Lemma A.1]{IRR}, we have that the process $e$ belongs to $\mathcal S^p(\P)$ for all $p \geq 1$, $e_{\tau^n} Y_{\tau^n}^2 $, $e_T |\zeta|^2$ and $\int_0^{T} 2 e_t |Y_t| |f(t,0,0)| dC_t$ are in $\mathbb L^p(\Q)$ for all $p \geq 1$.
In the same way we get the integrability of $\int_0^{\tau^n} 2 e_s |Y_s| d \Lambda_s$.
Hence,
$$ e_t Y_t^2
\leq \E^\Q\left[ e_{\tau_n} Y_{\tau_n}^2 +\int_0^{T} 2 e_s |Y_s| |f(s,0,0)|dC_s | \mathcal F_t \right], \quad t \leq \tau_n.$$
Now we let $n$ tend to infinity
$$e_t Y_t^{2} \leq \E^\Q\left[ e_{T} |\zeta|^2 +\int_0^{T} 2 e_s |Y_s||f(s,0,0)|dC_s | \mathcal F_t \right],$$
where we may apply the dominated convergence theorem.
The Young inequality with a constant $c_1>0$ gives
\begin{align*}
 Y_t^{2}
& \leq
\E^\Q\left[e_{T}|\zeta|^2 + \frac{1}{c_1} \sup_{t \in [0,T]} |Y_t|^2 + c_1 e_T^2 (\int_0^{T} |f(s,0,0)|dC_s)^2 | \mathcal F_t \right]
\\
& \leq \E^\Q\left[ \frac{1}{c_1} \sup_{t \in [0,T]} |Y_t|^2 + e^2_{T} \Theta_T | \mathcal F_t \right],
\end{align*}
where we set $\Theta_T=|\zeta|^2 +2c_1 (\int_0^{T} |f(s,0,0)|dC_s)^2$ and we take into account that $e_s/e_t \leq e_T$ for all $s,t \in [0,T]$ and $e_T \leq e_T^2$.
Let $p>1$, then we have
\begin{align*}
\sup_{t \in [0,T]} |Y_t|^{2p}
& \leq
\sup_{t \in [0,T]} \E^{\Q} \left[ \frac{1}{c_1} \sup_{t \in [0,T]} |Y_t|^2 + e^2_{T} \Theta_T | \mathcal F_t \right]^{p}.
\end{align*}
We apply Doob's inequality to obtain
\begin{align*}
\E^{\Q} \left[ \sup_{t \in [0,T]} |Y_t|^{2p} \right]
& \leq c \E^{\Q} \left[ \left( \E \left[ \frac{1}{c_1} \sup_{t \in [0,T]} |Y_t|^2 + e^2_{T} \Theta_T | \mathcal F_T \right]\right)^{p} \right]
\\
& \leq c \E^{\Q} \left[\frac{1}{c_1^p} \sup_{t \in [0,T]} |Y_t|^{2p} + e_{T}^{2p} \Theta_T^p \right],
\end{align*}
and choosing $c_1$ such that $c/c_1^p <1$ we have $\E^{\Q} \left[ \sup_{t \in [0,T]} |Y_t|^{2p} \right] \leq c \E^{\Q} \left[e_{T}^{2p} \Theta_T^p \right]$.
In Step 3 we give an estimate on $Z$ and $N$ under the measure $\Q$.
For $p > 1$ we deduce from the computations above that
\begin{align*}
&\left( \int_0^{\tau^n} |q_s Z_s^*|^2 dC_s \right)^p + \left( \int_0^{\tau^n} d[ N ]_s \right)^p\\
\leq & c \left(
|e_{\tau^n} Y_{\tau^n}^2|^p
+ \left| \int_0^{\tau^n}  e_s Y_s d \Lambda_s \right|^p
+ \left(\int_0^{T}  2 e_s |Y_s| |f(s,0,0)| dC_s \right)^p \right).
\end{align*}
Since the right hand side is exactly the same than in the proof of \cite[Lemma A.1]{IRR} we get following the same lines that
\begin{align}
& \E^{\Q} \left[\left( \int_0^{T} |q_s Z_s^*|^2 dC_s \right)^p + \left( \int_0^{T} d[ N ] _s \right)^p\right] \leq
c \E^{\Q}  \left[ |\zeta|^{2prk} + \left(\int_0^{T}  |f(s,0,0)| dC_s \right)^{2prk} \right]^{\frac{1}{rk}}
\label{zestimate}
\end{align}
where $r,k\geq 1$ are exponent coming from H\"older inequalities. We utilize the H\"older inequality with $rk$ to the estimate obtained for the process $Y$ in Step 2 and hence have
\begin{align}
\label{yestimate2}
\E^{\Q} \left[ \sup_{t \in [0,T]} |Y_t|^{2p} \right]
& \leq
c \E^{\Q}  \left[ |\zeta|^{2prk} + \left(\int_0^{T}  |f(s,0,0)| dC_s \right)^{2prk} \right]^{\frac{1}{rk}}.
\end{align}
\noindent
In step 4 we finally want to take the expectation under the measure $\P$.
Let us define $\hat M_t=M_t - \int_0^t H_s d [ M,M ]_s$ and note that since $H \cdot M$ is a BMO($\P$) martingale the process $H \cdot \hat M$ and hence $-H \cdot \hat M$ are BMO($\Q$) martingales (see \cite[Theorem 3.3]{Kazamaki}).
Furthermore by \cite[Theorem 3.1]{Kazamaki} there exist a $w, w' >1$ such that $\mathcal E(H \cdot M)_T \in L^w(\P)$ and $\mathcal E(- H \cdot \hat M)_T \in L^{w'}(\Q)$.
As $\mathcal{E}(H \cdot M)^{-1}=\mathcal E (-H \cdot \hat M)$ we have $d\P=\mathcal E (-H \cdot \hat M)_T d\Q$.
Now using the H\"older inequality with the conjugate exponent $v$ of $w$ (and $v'$ of $w$') and estimate \eqref{yestimate2} we deduce
\begin{align*}
 \E\left[ \sup_{t \in [0,T]} |Y_t|^{2p} \right]
& = \E^{\Q} \left[ \mathcal E (-H \cdot \hat M)_T \sup_{t \in [0,T]} |Y_t|^{2p} \right]
\\
& \leq \E^{\Q} \left[ \mathcal E (-H \cdot \hat M)_T^{w'} \right]^{\frac{1}{w'}} \E^{\Q} \left[ \sup_{t \in [0,T]} |Y_t|^{2pv'} \right]^{\frac{1}{v'}}
\\
& \leq c \left( \E^{\Q}  \left[ |\zeta|^{2pv'rk} + \left(\int_0^{T}  |f(s,0,0)| dC_s \right)^{2pv'rk} \right]^{\frac{1}{rk}} \right)^{\frac{1}{v'}}
\\
& \leq c \E  \left[ \mathcal E(H \cdot M)^w \right]^{\frac{1}{w}} \E \left[ |\zeta|^{2pvv'rk} + \left(\int_0^{T}  |f(s,0,0)| dC_s \right)^{2pvv'rk} \right]^{\frac{1}{rkvv'}} .
\end{align*}
Setting $q=vv'rk$ and treating estimate \eqref{zestimate} similarly gives the desired result. Remark that since $N$ does not appear in the definition of the measure $\Q$, Step 4 of this proof is exactly the same than the the same step in the proof of \cite[Lemma A.1]{IRR}.
\end{Proof}
\vspace{-1cm}
\section*{Acknowledgments}
The authors is very grateful to DFG research center Matheon project E2 for financial support.

\end{document}